\documentclass[11pt]{amsart}

\usepackage{eucal,amsrefs,graphicx,amssymb}
\usepackage[all, knot]{xy}

\theoremstyle{definition}
\newtheorem{ex}{Example}[section]
\newtheorem*{ex*}{Example}
\newtheorem*{defn}{Definition}

\theoremstyle{plain}
\newtheorem{thm}{Theorem}[section]
\newtheorem{prop}[thm]{Proposition}
\newtheorem*{cor*}{Corollary}
\newtheorem{cor}[thm]{Corollary}
\newtheorem{lem}[thm]{Lemma}

\newtheorem{thm*}{Theorem}

\theoremstyle{remark}
\newtheorem*{rem}{Remark}

\numberwithin{equation}{section}

\DeclareMathOperator{\spec}{Spec}

\DeclareMathOperator{\End}{End}
\DeclareMathOperator{\Hom}{Hom}
\DeclareMathOperator{\Pic}{Pic}

\DeclareMathOperator{\WDiv}{WDiv}
\DeclareMathOperator{\CDiv}{CDiv}
\DeclareMathOperator{\Div}{div}
\DeclareMathOperator{\Deg}{Deg}
\DeclareMathOperator{\lcm}{lcm}

\def\C{\mathbb{C}}
\def\Q{\mathbb{Q}}
\def\P{\mathbb{P}}
\def\R{\mathbb{R}}
\def\Z{\mathbb{Z}}
\def\P{\mathbb{P}}

\def\N{\mathbb{N}}

\def\M{{\bf{M}}}

\def\i{{\texttt{\textit i}}}
\def\WDivT{\WDiv_T}
\def\CDivT{\CDiv_T}
\def\fan{\Sigma}
\def\z{{\bf z}}
\providecommand{\norm}[1]{\lVert#1\rVert}

\begin{document}

\title[Algebraic Stability and Degree Growth]
{Algebraic Stability and Degree Growth of \\
               Monomial Maps and Polynomial Maps}

\author{Jan-Li Lin}

\address{Department of Mathematics, Indiana University, Bloomington \\ IN 47405 \\ USA}

\email{janlin@indiana.edu}



\begin{abstract}
Given a rational monomial map, we consider the question of finding a toric variety on which it
is algebraically stable. We give conditions for when such variety does or does not exist.
We also obtain several precise estimates of the degree sequences of monomial maps on $\P^n$.
Finally, we characterize polynomial maps which are algebraically stable on $(\P^1)^n$.
\end{abstract}

\maketitle

\tableofcontents
%%%%%%%%%%%%%%%%%%%%%%%%%%%%%%%%%%%%%%%%%%%%%%%%%%%%%%%%%%%%%%%%%%%%%%
%%%%
%%%%   Section:  Introduction
%%%%
%%%%%%%%%%%%%%%%%%%%%%%%%%%%%%%%%%%%%%%%%%%%%%%%%%%%%%%%%%%%%%%%%%%%%%

\section{Introduction}
We study the dynamical behavior of two family of maps, namely, monomial maps and polynomial maps.
In particular, we focus on two aspects: algebraic stability and degree growth.
For monomial maps, we use the theory of toric varieties as
the main tool. For polynomial maps, we focus on their dynamical behavior on the space
$(\P^1)^n=\underbrace{\P^1\times\cdots\times\P^1}_{n}\,$.

Given an $n\times n$ integer matrix $A=(a_{i,j})$, there is an associated monomial map $f_A:\C^n\to\C^n$ defined by
\[
\textstyle{f_A(x_1,\cdots,x_n)=( \prod_j x_j^{a_{1,j}},\cdots ,\prod_j x_j^{a_{n,j}} ).  }
\]
Monomial maps fit nicely into the framework of toric varieties and equivariant maps (also called toric maps) on them. In this paper,
we try to make extensive use of the toric method to study the dynamics of monomial maps.

The idea of applying the theory of toric varieties to monomial maps is in fact not new.
For example, Favre~\cite{Fa} used the
orbit-cone correspondence of the torus action to translate a criterion of algebraic stability
to a condition about cones in a fan,
 and uses it to classify monomial maps in the case of toric surfaces.
In order to generalize his result to higher dimension, one needs a good understanding on pulling back
cohomology classes under rational maps. So we start from
a formula of pulling back divisors in toric varieties (Theorem ~\ref{thm:pullback}).

We then define the notion of algebraic stability and prove a criterion similar to the one in ~\cite{Fa}.
Results about stability are proven using the criterion.
For example, we proved that every monomial polynomial map is algebraically stable on $(\P^1)^n$. Also,
we generalize some results of ~\cite{Fa} to higher dimension.

\setcounter{section}{5}
\setcounter{thm}{6}
\begin{thm}
Suppose that $A\in \M_n(\Z)$ is an integer matrix.
\begin{enumerate}
\item If there is a unique eigenvalue $\lambda$ of $A$ of maximal modulus, with algebraic multiplicity one;
      then $\lambda\in\R$, and there exists a simplicial toric birational model X (maybe singular)
and a $k\in\N$ such that $f_A^k$ is strongly algebraically stable on $X$.
\item If $\lambda,\bar{\lambda}$ are the only eigenvalues of $A$ of maximal modulus, also with algebraic
      multiplicity one, and if
      $\lambda= |\lambda|\cdot e^{2\pi\i\theta}$, with $\theta\not\in\Q$; then there is no toric birational model
      which makes $f_A$ strongly algebraically stable.
\end{enumerate}
\end{thm}

For the definition of (strongly) algebraically stable, see section~\ref{sec:alg_stab}.
We note that many of the results concerning stabilization of monomial maps
in this paper have been obtained independently by Mattias Jonsson
and Elizabeth Wulcan ~\cite{JW}.

Next, we focus on two spaces: the projective space $\P^n$, and the product of the projective line
$(\P^1)^n$.
For the projective space $\P^n$, the monomial map
$f_A$ induces a rational map $\P^n\dashrightarrow\P^n$, also denoted by $f_A$.
The pull back $f^*$ of a rational map $f:\P^n\dashrightarrow\P^n$ on $H^{1,1}(\P^n;\R)$
is given by the degree of $f$.
 Thus we consider the degree sequence
$\{\deg(f_A^k)\}_{k=1}^\infty$.
Results about the degree sequence of monomial maps can be found in ~\cite{BK} and ~\cite{HP}.

In particular, one can define the {\em asymptotic degree growth}
\[
\delta_1(f_A)=\lim_{k\to\infty} (\deg(f_A^k))^{\frac 1 k}.
\]
Hasselblatt and Propp (\cite[Theorem 6.2]{HP}) proved that $\delta_1(f_A)=\rho(A)$, the spectral radius of the matrix $A$.
We refine the above result and obtain the following description of the asymptotic behavior of the degree sequence for a general monomial map.

\setcounter{section}{6}
\setcounter{thm}{1}
\begin{thm}
Given an $n\times n$ integer matrix $A$ with nonzero determinant, assume that $\rho(A)$ is the
spectral radius of $A$.
Then there exist two positive constants $C_1\ge C_0>0$
and a unique integer $\ell$ with $0\le\ell\le n-1$, such that
\[
C_0\cdot k^\ell\cdot\rho(A)^k \le \deg(f_A^k) \le C_1\cdot k^\ell\cdot\rho(A)^k
\]
for all $k\in\N$.

In fact, $(\ell+1)$ is the size of the largest Jordan block of $A$ among the ones corresponding
to eigenvalues of maximal modulus.
\end{thm}

Moreover, if the matrix $A$ has some better property, then we can describe the degree sequence even
more precisely. This is the content of Theorem~\ref{thm:degseq1}, Theorem~\ref{thm:degseq2}, and the following
theorem.

\setcounter{thm}{7}
\begin{thm}
Assume that the matrix $A$ is diagonalizable,
and assume for each eigenvalue $\lambda$ of $A$ of maximum modulus, $\lambda/\bar{\lambda}$ is a root of unity.
Then there is a positive integer $p$, and $p$ constants $C_0,C_1,\cdots, C_{p-1}\ge 1$, such that
\[
\deg(f_A^{pk+l})=C_l\cdot |\lambda_1|^{pk+l} + O(|\lambda_2|^{pk+l}),
\]
where $l=0,1,\cdots,p-1$.
\end{thm}

Let us mention that the above theorems about the degree sequences of monomial maps
can be generalized to the case of weighted projective spaces.
On weighted projective spaces, we have the notion of {\em weighted degree} of a toric map, and their growth
under iterations follows the pattern as the degree growth of monomial maps in projective spaces.
This generalization is
suggested to us by Mattias Jonsson. We introduce weighted projective space briefly,
and explain the generalization in $\S$\ref{subsec:WPS}.

On $(\P^1)^n$, we obtain a concrete matrix representation for the pull back on the
Picard groups for general rational maps.
We apply the matrix representation to give another proof of the above theorem
of Hasselblatt and Propp about the first dynamical degree
of a monomial map (Theorem~\ref{thm:d1_mono}).

In the last subsection ($\S$\ref{subsec:stab_poly}) of this paper,
we study the stability of polynomial maps on $(\P^1)^n$.
As a result, we obtain the following characterization:

\setcounter{section}{7}
\setcounter{thm}{4}

\begin{thm}
Let $f=(f_1,\cdots,f_n)$ be a polynomial map.
\begin{enumerate}
\item
If each $f_j$ is dominated by a monomial term, the $f$ is
algebraically stable on $(\P^1)^n$.
\item
Assume that, for some iterate $f^N=(f_1^{(N)},\cdots,f_n^{(N)})$ of $f$, we have $\deg_{z_i}(f^{(N)}_j)>0$  for all $i,j=1,\cdots,n$.
Then $f$ being algebraically stable on $(\P^1)^n$ implies that each $f_j$ must have a dominant term.
\end{enumerate}
\end{thm}

Here we say that a polynomial $f_j(z_1,\cdots,z_n)$ is dominated by the monomial $\mu$ if the coefficient of $\mu$
in $f_j$ is non-zero, and $\deg_{z_i}(f_j)=\deg_{z_i}(\mu)$ for all variables $z_i$, $i=1,\cdots,n$.

%This paper is organized as follows. In section 2, we introduce the basic set-up and notations for
%toric varieties, then
%we discuss the maps between toric varieties in section 3. Divisors on toric varieties is introduced in Section 4,
%where a formula of pulling back divisors in the toric setting is given.
%Then we define algebraic stability in section 5 and discussed criteria for stability.
%Some results about monomial maps on $\P^n$ are given in section 6.
%Finally, we study monomial maps and polynomial maps on $(\P^1)^n$ in Section 7.

{\bf Acknowledgement.}
The author would like to thank Professor Eric Bedford for suggesting
the problem and for many invaluable discussions. Without
Professor Bedford's continuous support this work would have been impossible.
He would also like to thank Jeff Diller, Mattias Jonsson, and Charles Favre
for useful discussions and suggestions.

\setcounter{section}{1}

%%%%%%%%%%%%%%%%%%%%%%%%%%%%%%%%%%%%%%%%%%%%%%%%%%%%%%%%%%%%%%%%%%%%%%
%%%%
%%%%   Section:  Toric Varieties
%%%%
%%%%%%%%%%%%%%%%%%%%%%%%%%%%%%%%%%%%%%%%%%%%%%%%%%%%%%%%%%%%%%%%%%%%%%

\section{Toric varieties}

In this section, we give a brief survey of basic definitions and properties of toric varieties.
For more detail, we refer the readers to ~\cite{Dan} or ~\cite{Fu}.

\subsection{Cones and affine toric varieties}

Let $N\cong\Z^n$ be a lattice of rank $n$, and $N_\R := N\otimes_\Z \R\cong\R^n$ be the associated real
vector space. A (convex) {\em polyhedral cone} in $N_\R$ is a subset of the form
\[
\sigma = \{ \sum_{i=1}^k r_i v_i\ |\ r_i\in\R_{\ge 0}, v_i\in N_\R \}
\]
for some finite set of vectors $v_1,\cdots,v_k$. In the case $k=0$, we make the convention that
$\sigma =\{0\}$, the cone containing only the origin.

The {\em dimension} of a cone is the dimension of the $\R$-linear subspace spanned by the generating set.
A cone is {\em strongly convex} if it does not contain any line through the
origin. A cone is {\em rational} if we can choose the generators $v_1,\cdots,v_k$ from the lattice $N$.
In what follows, by a cone we always mean ``a strongly convex, rational polyhedral cone''.

From the lattice $N$ we can form the {\em dual lattice} $M:=\Hom_\Z(N,\Z)$, with dual pairing denoted by
$\langle\ ,\, \rangle$. It is a lattice in the
dual vector space
\[
M_\R:=M\otimes_\Z \R\cong \Hom_\R(N_\R,\R)=N_\R^\vee.
\]

The {\em dual cone}
$\sigma^\vee$ of $\sigma$ is defined by
\[
\sigma^\vee = \{ u\in M_\R\ |\ \langle u,v \rangle \ge 0 \text{ for all } v\in\sigma\}.
\]

The intersection $S_\sigma=\sigma^\vee\cap M$ is a finitely generated monoid by Gordan's lemma.
The affine variety $U_\sigma:=\spec( \C[S_\sigma] )$ of the ring $\C[S_\sigma]$ is called
the {\em affine toric variety} associated to the cone $\sigma$. More concretely, a closed point in $U_\sigma$
corresponds to a semigroup morphism $(S_\sigma,+)\to (\C,\cdot)$ which sends $0\in S_\sigma$ to $1\in\C$.

\begin{ex}
\label{ex:C2}
Let $N=\Z^2$, and let $\sigma$ be the cone in $N_\R\cong\R^2$ generated by
$e_1=(1,0)$ and $e_2=(0,1)$.
It is easy to see that $S_\sigma$ is the monoid generated by the dual basis $e_1^*, e_2^*$, and
$\C[S_\sigma]\cong \C[x,y]$.
Thus the affine toric variety
$U_\sigma\cong\spec\C[x,y]=\C^2$.

\end{ex}

\begin{ex}
More generally, let $N=\Z^n$, and let $\sigma$ be the cone in $N_\R\cong\R^n$ generated by
the standard basis $e_1,\cdots, e_n$. Then the affine toric variety
$U_\sigma\cong\C^n$.
\end{ex}

Now we are going to introduce some definitions about cones.
\begin{itemize}
\item One dimensional cones are also called
{\em rays}. On each ray, there is a unique nonzero integral point of the smallest norm ; it is called {\em the ray generator}.

\item A cone is {\em simplicial} if it is generated by linearly independent vectors.

\item A cone is {\em smooth} if it is generated by part of a basis for the lattice $N$.
\end{itemize}
We remark here that a cone $\sigma$ is
smooth if and only if the corresponding affine variety $U_\sigma$ is smooth.

\subsection{Fans and general toric varieties}
A fan $\fan$ in $N_\R$ is a set of cones in $N_\R$ satisfying the following two conditions:
\begin{enumerate}
\item each face of a cone in $\fan$ is a cone in $\fan$.
\item the intersection of two cones in $\fan$ is a face of each (hence also in $\fan$).
\end{enumerate}
Unless otherwise stated, we will assume that fans are {\em finite}, i.e., they contain finitely many cones.

From a fan $\fan$, we can construct the {\em toric variety} $X(\fan)$ corresponding to $\fan$.
First, we take the
disjoint union $\coprod_{\sigma\in\fan} U_\sigma$, then we glue them as follows. For cones
$\sigma,\tau\in\fan$, the intersection $\sigma\cap\tau$ is a face of each. Thus we have open immersions
\[
U_\sigma \hookleftarrow U_{\sigma\cap\tau} \hookrightarrow U_\tau.
\]
For each pair of cones $\sigma, \tau \in \fan$, we glue $U_\sigma$ and $U_\tau$
along the open subvariety $U_{\sigma\cap\tau}$. The gluing data
will be compatible, and the resulting variety is the toric variety $X(\fan)$.

We use the notion $\fan(k)$ to denote the set of all $k$-dimensional cones of $\fan$.
The {\em support} of a fan $\fan$, denoted $|\fan|$, is the union of its cones, i.e.,
$|\fan|:=\cup_{\sigma\in\fan} \sigma$. A fan $\fan$ is {\em complete} if $|\fan|=N_\R$.
A fan is complete if and only if the corresponding toric variety is a complete variety, i.e.,
the underlying topological space is compact in the classical topology.
For a fan $\fan$, the toric variety $X(\fan)$ is smooth if and only if every cone in $\fan$ is smooth.
To check smoothness, it is enough to check for all {\em maximal cones} in $\fan$, i.e., cones that are
not proper faces of any other cone.

\begin{ex}
Let $N=\Z^2$, and let $e_0=(-1,-1)=-(e_1+e_2)$. Let $\tau_i$ be the ray generated by $e_i$ for $i=0,1,2$,
and $\sigma_{ij}$ be the cone generated by $e_i$ and $e_j$. The set
\[
\fan = \{\ \{0\}, \tau_0, \tau_1, \tau_2, \sigma_{12}, \sigma_{02}, \sigma_{01}\ \}
\]
is a fan, and the corresponding toric variety is the projective plane $\P^2$. In fact, if we set the homogeneous
coordinate of $\P^2$ as $[w;x;y]$, then we can identify $U_{\sigma_{12}}$ as
$\spec\C[\frac x w,\frac y w]\cong\{[w;x;y]\in\P^2 | w\ne 0\}$. Similarly,
$U_{\sigma_{02}}\cong \spec\C[\frac y x,\frac w x]\cong\{[w;x;y]\in\P^2 | x\ne 0\}$, and
$U_{\sigma_{01}}\cong \spec\C[\frac x y,\frac w y]\cong\{[w;x;y]\in\P^2 | y\ne 0\}$.

\begin{figure}
\centerline{  \includegraphics{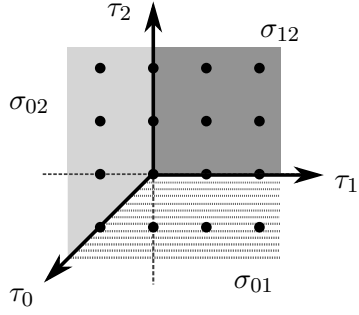} }
\begin{picture}(0,0)(0,0)
\put(-65,7){\makebox{$\tau_0$} }
\put(58,50){\makebox{$\tau_1$} }
\put(30,107){\makebox{$\sigma_{12}$} }
\put(-28,115){\makebox{$\tau_2$} }
\put(20,12){\makebox{$\sigma_{01}$} }
\put(-65,80){\makebox{$\sigma_{02}$} }
\end{picture}
\caption{The fan structure of $\P^2$.}
\end{figure}
\end{ex}

\begin{ex}
More generally, let $N=\Z^n$, $e_1,\cdots,e_n$ be the standard basis of $N$, and $e_0=-(e_1+\cdots +e_n)$.
For any proper subset $I$ of the set ${\bf n}=\{ 0,1,\cdots,n\}$, i.e., $I\subsetneq {\bf n}$,
let $\sigma_I$ be the cone generated
by $\{e_i|i\in I\}$.
Then the set $\fan=\{\sigma_I| I\subsetneq {\bf n}  \}$ forms a fan. The toric variety associated to the fan is the
projective space $X(\fan)\cong \P^n$.
\end{ex}

\begin{ex}
\label{ex:P1n}
In this example, we will construct a fan $\fan$ that corresponds to the product of the projective line
$(\P^1)^n=\underbrace{\P^1\times\cdots\times\P^1}_{n}\,$. There are $2n$ rays in $\fan$. They are generated
by the standard basis vectors $e_1,\cdots,e_n$ and their negatives $-e_1,\cdots,-e_n$.
The maximal cones, i.e., $n$-dimensional cones, are generated by vectors of the form
$\{ (s_1 e_1),\cdots, (s_n e_n) \}$, where $s_i\in\{+1,-1\}$ are the signs. All other cones in $\fan$
are faces of some maximal cone. For example, the fan for $\P^1\times\P^1$ has four maximal cones, generated by
$\{e_1,e_2\} , \{-e_1,e_2\}, \{e_1,-e_2\}$, and $\{-e_1,-e_2\}$, respectively.
\end{ex}

\subsection{The orbits of the torus action}

Every toric variety $X(\fan)$ is equipped with a torus action, thus $X(\fan)$ can be written
as the disjoint union of the orbits.
The orbits are in 1-1 correspondence with cones in the fan $\fan$ in a nice way, as follows.
First, for each cone $\tau\in\fan$, there is a distinguished point $x_\tau\in U_\tau$.
It is the closed point corresponding to the following semigroup morphism $S_\sigma\to\C$.
\[
u\longmapsto \begin{cases} 1  &   \text{if  $u\in\sigma^\bot$,} \\
                       0  &   \text{otherwise.}  \end{cases}
\]
Then we define $O_\tau$ to be the orbit of $x_\tau$ under the torus action. The only closed orbits are
fixed points, which correspond to the $n$-dimensional cones. We can take the closure of an orbit
in the toric variety, it is called the {\em orbit closure} of $\tau$, denoted by $V(\tau)$. An orbit closure
itself is a toric variety, it is a closed, toric subvariety of the original toric variety.

%%%%%%%%%%%%%%%%%%%%%%%%%%%%%%%%%%%%%%%%%%%%%%%%%%%%%%%%%%%%%%%%%%%%%%
%%%%
%%%%  Section: Toric Maps
%%%%
%%%%%%%%%%%%%%%%%%%%%%%%%%%%%%%%%%%%%%%%%%%%%%%%%%%%%%%%%%%%%%%%%%%%%%

\section{Toric maps}

Suppose $A:N\to N'$ is a homomorphism of lattices, $\fan$ is a fan in $N$, and $\fan'$ is a fan in $N'$.

\begin{defn}
Given a cone $\sigma\in\fan$, we say that $\sigma$ {\em maps regularly} to $\fan'$ by $A$ if there is a cone
$\sigma'\in\fan'$ such that $A(\sigma)\subseteq\sigma'$. In this case, we call the smallest such cone in $\fan'$
the {\em cone closure} of the image of $\sigma$, and denote it by $\overline{A(\sigma)}$.
\end{defn}

If $A:N\to N'$ is a homomorphism such that every cone of $\fan$ maps regularly to $\fan'$, then $A$
induces a morphism of varieties $f_A:X(\fan)\to X(\fan')$.
Furthermore, $f_A$ will be {\em equivariant}
under the torus action. Conversely, every equivariant morphism $X(\fan)\to X(\fan')$ is induced by a
homomorphism of lattices satisfying the above property.
Equivariant morphisms will map orbits to orbits. If $\sigma\in\fan$, then $f_A$ maps $O_\sigma\subset X(\fan)$ to
$O_{\overline{A(\sigma)}}\subset X(\fan')$.

More generally, any homomorphism of lattices $A:N\to N'$ induces an {\em equivariant rational map}
$f_A:X(\fan)\dashrightarrow X(\fan')$. On a complete toric variety, $f_A$ is {\em dominant} if and only if
$A_\R=(A\otimes\R):N_\R\to N'_\R$ is surjective.
In this paper, we will study the dynamics of dominant, toric rational self maps on complete toric varieties.

\begin{ex}
Let the matrix
$A=\bigl( \begin{smallmatrix} a&b\\ c&d \end{smallmatrix} \bigr)\in M_2(\Z) $, and identify
$N\cong\Z^2$ as column vectors.
Then $A:N\to N$ is given by multiplying column vectors by $A$ on the left.
Let us see what $f_A$ does on the affine toric variety $\C^2$ we constructed in Example~\ref{ex:C2}.
Notice that $A$ induces the linear map $^t A$ on $M$, which
sends $e_1^*\mapsto a\cdot e_1^*+b\cdot e_2^*$ and $e_2^*\mapsto c\cdot e_1^*+d\cdot e_2^*$.
The isomorphism $\C[S_\sigma]\cong\C[x,y]$ is
given by $e_1^*\mapsto x$ and $e_2^*\mapsto y$, thus $f_A$ is the map $x\mapsto x^a y^b$ and $y\mapsto x^c y^d$.
\end{ex}

\begin{ex}
Let $A=(a_{i,j})$ be an $n\times n$ integer matrix, then using a similar argument as the above example, we
can see that the map $f_A:\C^n \dashrightarrow \C^n$ is given by
\[
\textstyle{f_A(x_1,\cdots,x_n)=( \prod_j x_j^{a_{1,j}},\cdots ,\prod_j x_j^{a_{n,j}} ).  }
\]
So we know on $\C^n$, toric rational self maps are exactly monomial maps.
\end{ex}

\subsection{Toric endomorphisms}
Before we start to study the dynamics of toric rational self maps, one might ask:
what do we know about toric morphisms from a toric variety to itself?
The following property provides an answer. It turns out that those morphisms have
quite simple structure.

Let $\fan$ be a complete fan in $N_\R$.  Let $A :N\to N$ be a homomorphism of lattice that maps each cone of $\fan$
regularly to $\fan$, then $f_A:X(\fan)\to X(\fan)$ is a toric morphism. Also assume that $f_A$ is dominant, i.e.
$A_\R: N_\R\to N_\R$ is surjective. Moreover, suppose that $v_1,\cdots,v_d$ are the ray generators of $\fan$.

\begin{prop}
There is a positive integer $k$ such that $A^k v_i = a_i v_i$ for some positive integers $a_i$,
$i=1,\cdots,d$. In particular, the iteration $f_A^k:X(\fan)\to X(\fan)$ maps every orbit onto itself.
\end{prop}

\begin{proof}
We know that $A$ maps each cone $\sigma$ into some other cone $\sigma'$. Since $A_\R$ is
a surjective endomorphism of $N_\R$, it is indeed an $\R$-linear automorphism. Thus $A_\R$ will preserve
the dimension of cones in $\fan$. As a consequence, what $A_\R$ does on the cones of $\fan$ is
just permuting them (preserving the dimension). Therefore, for some integer $k$, $A^k_\R$ would fix
every cone in $\fan$. In particular, for one dimensional cones, we know that
$A^k_\R(\R_{\ge 0}\cdot v_i)=\R_{\ge 0}\cdot v_i$.
Furthermore, since $A:N\to N$ and $v_i\in N$, we deduce that $A(v_i)=a_i v_i$ for some positive
integer $n_i$.
Finally, the last sentence is an immediate consequence of the first sentence and the description of
images of orbits under a morphism.
\end{proof}

\begin{ex}
Given non-zero integers $a_i$, the endomorphism of $(\P^1)^n = \P^1\times\cdots\times\P^1$
coming from the monomial map of $\C^n$ of the form
$f(x_1,x_2\cdots,x_n)=(x_1^{a_1}, x_2^{a_2},\cdots, x_n^{a_n})$
is toric. Conversely, every toric endomorphism of $(\P^1)^n$ is of the form
\[
f(x_1,x_2\cdots,x_n)=(x_{s(1)}^{a_1}, x_{s(2)}^{a_2},\cdots, x_{s(n)}^{a_n})
\]
for some permutation $s\in S_n$.
\end{ex}

The above example shows that, in general, the above proposition cannot be improved. However, if the space
$X(\fan)$ is nice, there are possibly stronger condition on the map $A$, as we can see in the following example.

\begin{ex}
For the space $\P^n$, we know the one dimensional cones are generated by $e_0,e_1\cdots,e_n$. Every $n$ of
them would form a basis for $N$, and $e_0+e_1+\cdots +e_n=0$. Applying $A^k$, we get
$a_0 e_0 + a_1 e_1+\cdots+ a_n e_n=0$. This will force $a_0=a_1=\cdots=a_n$, and hence we know
$f_A^k([x_0:\cdots:x_n])=[x_0^a:\cdots:x_n^a]$ for some positive integer $a$.
\end{ex}

%%%%%%%%%%%%%%%%%%%%%%%%%%%%%%%%%%%%%%%%%%%%%%%%%%%%%%%%%%%%%%%%%%%%%%
%%%%
%%%%  Section: Divisors on Toric Varieties
%%%%
%%%%%%%%%%%%%%%%%%%%%%%%%%%%%%%%%%%%%%%%%%%%%%%%%%%%%%%%%%%%%%%%%%%%%%

\section{Divisors on toric varieties}

Divisors are the main tool in the study of codimension-one geometry of varieties.
Since we work on toric varieties, the divisors that are invariant under the torus action are especially important.
We will recall basic definitions and properties of divisors in a toric variety, then prove a formula about
pulling back divisors.

\subsection{Weil Divisors and divisor class groups}

In a toric variety, the torus invariant prime Weil divisors are exactly the codimension one orbit closures,
i.e., $V(\tau)$ for $\tau\in\fan(1)$. Let $\fan(1)=\{\tau_1,\cdots,\tau_d\}$ be the (finite) set of rays in $\fan$,
a $T$-invariant Weil divisor, $T$-Weil divisor for short,
is then of the form $\sum_{i=1}^d a_i V(\tau_i)$ where $a_i\in\Z$.
The group of $T$-Weil divisors, denoted $\WDivT(X(\fan))$, then equals to $\oplus_{i=1}^d \Z\cdot V(\tau_i)$,
the free abelian group generated by the $T$-invariant prime divisors. The principal divisors in
$\WDivT(X(\fan))$ are in 1-1 correspondence to elements of $M$ in the following way.
For each element $u\in M$, there is a corresponding character $(\C^*)^n\to\C^*$ which extend to a global meromorphic
function $\chi^u$ on $X(\fan)$. This gives the principal divisor
\[
\Div(\chi^u)=\sum_{i=1}^d \langle u,v_i \rangle V(\tau_i),
\]
where $v_i$ is the ray generator of $\tau_i$.
Conversely, every principal divisor in
$\WDivT(X(\fan))$ is of the form $\Div(\chi^u)$ for some $u\in M$. Hence we can identify $M$ as a subgroup
of $\WDivT(X(\fan))$, and the quotient
\[
A_{n-1}(X(\fan)) := \WDivT(X(\fan)) / M
\]
is the {\em divisor class group} of $X(\fan)$.

\subsection{Cartier divisors and Picard groups}

In a complete toric variety $X(\fan)$ associated to a complete fan $\fan$, the torus invariant Cartier divisors,
or for simplicity, $T$-Cartier divisors, is given by the following data. For each cone $\sigma\in\fan(n)$ of
maximal dimension, we specify an element $u(\sigma)\in M$. The datum
$\{u(\sigma) | \sigma\in\fan(n)\}$ are required to satisfy the
compatibility condition that $[u(\sigma)]=[u(\sigma')]$ in $M/M(\sigma\cap\sigma')$, where
$M(\sigma\cap\sigma')=(\sigma\cap\sigma')^\bot\cap M$. We write $D=\{u(\sigma)\}$ and call it the Cartier divisor
defined by the data $\{u(\sigma)\}$. We denote the group of all T-Cartier divisor by $\CDivT(X(\fan))$.

Each $u(\sigma)$ defines a $T$-Weil divisor $\Div(\chi^{-u(\sigma)})$ on $U_\sigma$ (the negative sign here is
to be consistent
with the literature). The compatibility condition means these divisors agree on overlaps, thus every $T$-Cartier
divisor $D=\{ u(\sigma) \}$ gives rise to a unique $T$-Weil divisor
\[
[D] = \sum_{\tau_i\in\fan(1)} -\langle u(\sigma), v_i \rangle \cdot V(\tau_i),
\]
here $\sigma$ in the summand is any maximal cone such that $\tau_i\in\sigma$. The sum is independent of the choice
of $\sigma$ by the compatibility condition.

The data $\{u(\sigma) | \sigma\in\fan(n)\}$ also defines a continuous piecewise linear function $\psi_D$ on $N_\R$.
The restriction of $\psi_D$ to the maximal cone $\sigma$ is given by $u(\sigma)$, i.e.,
$\psi_D(v)=\langle u(\sigma),v\rangle$ for $v\in\sigma$. The continuity comes from the compatibility.
Conversely, a continuous piecewise linear function $\psi$ on $N_\R$, which is also integral on each cone
(i.e., given by an element of the lattice $M$ on each cone),
determines a unique $T$-Cartier divisor $D$, with $[D]= \sum -\psi(v_i)\cdot V(\tau_i)$.
The function $\psi$ is called the {\em support function} of the Cartier divisor $D$.
On a complete toric variety, a $T$-Cartier divisor is {\em ample} if and only if its support function
is strictly convex(\cite[p.70]{Fu}).

There is a natural way to identify $M$ as a subgroup of $\CDivT(X(\fan))$. Each $u\in M$ is identified with
the Cartier divisor such that $u(\sigma)=u$ for all $\sigma\in\fan(n)$.
The Weil divisor of this Cartier divisor is exactly the principal divisor defined by $\chi^u$.
The quotient $\CDivT(X(\fan))/M$ is the {\em Picard group} of $X(\fan)$, and is denoted by $\Pic(X(\fan))$.

We conclude this section by mentioning relations between Picard groups and cohomology groups. For a
complete toric variety $X$, we have $\Pic(X)=H^2(X;\Z)$. If $X$ is also simplicial, then
\[
H^{1,1}(X):=H^1(X,\Omega_X)=H^2(X;\C)=\Pic(X)\otimes_\Z \C.
\]

\subsection{Pulling back divisors}

Assume that $\fan, \fan'$ are complete fans in $N_\R, N'_\R$, respectively, which associates to
the complete toric varieties $X(\fan)$ and $X(\fan')$.
Let $A$ be a homomorphism $N\to N'$ such that $A\otimes_\Z\R$ is surjective. It induces
a dominant rational map
$f_A: X(\fan)\dashrightarrow X(\fan')$. We would like to study the pull back of a Cartier divisor
on $X(\fan')$, which gives, in general, a Weil divisor of $X(\fan)$.

A Cartier divisor on $X(\fan')$ corresponds to a unique integral piecewise linear function on $\fan'$.
Let $D$ be a Cartier divisor, with support function $\psi_D$.

\begin{thm}\label{thm:pullback}
Let $\fan,\fan'$ be complete fans, and $f_A:X(\fan)\dashrightarrow X(\fan')$ be a dominant toric rational
map induced by $A:N\to N'$.
The pull back of $D$ via $f_A$ is
\[
f_A^*D= \sum_{\tau_i\in\fan(1)} -\psi_D (A v_i)\cdot V(\tau_i).
\]
Here the $\tau_i$ run through all one-dimensional cones of $\fan$, and $v_i$ is the ray generator
of the ray $\tau_i$.
\end{thm}

\begin{proof}
We can refine the fan $\fan$ to get a fan $\widetilde{\fan}$ such that $A$ induces a toric morphism
from $X(\widetilde{\fan})$ to $X(\fan)$. In order to distinguish from $f_A$,  we call this morphism
$\widetilde{f}_A$. The morphism $\pi: X(\widetilde{\fan}) \to X(\fan)$ is induced by the identity map
on $N$. It is proper and birational. So we have the following diagram.

\[
\xymatrix{
& X(\widetilde{\fan})\ar[ld]_{\pi}\ar[rd]^{\widetilde{f}_A}\\
X(\fan)\ar@{-->}[rr]_{\widetilde{f}_A\circ\pi^{-1}} & & X(\fan')
}
\]

We are going to pull back the divisor $D$ by first pull it back by $\widetilde{f}_A$, then push it forward
by $\pi$, i.e., $f_A^* D = \pi_* ( \widetilde{f}_A^* D ) $. Notice that once we show $f_A^*$ is given by the
above expression, then since the expression is independent of the refined fan $\widetilde{\fan}$,
so is $f_A^*$.

A Cartier divisor $D=\{u(\sigma); \sigma\in\fan(n)\}$ is locally cut out by the equation $\chi^{u(\sigma)}=0$ on $U_\sigma$.
The map $A:N\to N'$ induces the map $M'\to M$ defined by $u\mapsto u\circ A$. Therefore, the pull back
of $D$ under the morphism $\widetilde{f}_A$ is given by $\widetilde{f}_A^*D=\{u(\sigma)\circ A\}$.
We can also describe the pull back $\widetilde{f}_A^*D$ using
its support function. If $\psi_D$ is the support function of $D$, then $\widetilde{f}_A^* D$ will
have $\psi_D\circ A$ as its support function. Thus, as a Weil divisor, we have
\[
\widetilde{f}_A^* D = \sum_{\tau_i\in\widetilde{\fan}(1)} -\psi_D (A v_i)\cdot V(\tau_i).
\]

The fan $\widetilde{\fan}$ is a subdivision of $\fan$, and $\pi$ is just the toric morphism induced by identity.
The push forward map $\pi_*$ is given, on the prime divisors $V(\tau)$ for $\tau\in \widetilde{\fan}(1)$, by
\[
\pi_* V(\tau)=\begin{cases} V(\tau)&   \text{if $\tau\in\fan(1)$,} \\
                                0  &   \text{if $\tau\not\in\fan(1)$.}  \end{cases}
\]

Therefore, combining the two steps, we obtain
\[
f_A^* D = \pi_* (\widetilde{f}_A^* D) = \sum_{\tau_i\in\fan(1)} -\psi_D (A v_i)\cdot V(\tau_i).
\]
\end{proof}

Notice that, if $D$ is a principal divisor, i.e., $\psi_D$ is a linear function, then the pull back $f_A^* D$
will again be principal, given by the linear function $\psi_D\circ A$. Thus it induces a map, also denoted
by $f_A^*$, from the Picard group to the divisor class group.
\[
f_A^*: \Pic(X(\fan'))\to A_{n-1}(X(\fan)).
\]

If the fan is smooth, then so is the toric variety, and the notions of Weil
divisors and Cartier divisors coincide. So the pull back map of a Cartier divisor is still Cartier, and
it induces a map on Picard groups.
\[
f_A^*: \Pic(X(\fan'))\to \Pic(X(\fan)).
\]

 In a simplicial toric variety, every Weil divisor $D$ is
$\Q$-Cartier, i.e., some positive integral multiple of $D$ is Cartier.
Thus if we denote $G_\Q=G\otimes_\Z \Q$ for an abelian group $G$, and assume that $\fan,\fan'$ are simplicial,
then we have both maps
\begin{align*}
f_A^*&: \CDivT(X(\fan'))_\Q\to \CDivT(X(\fan))_\Q,\\
f_A^* &: \Pic(X(\fan'))_\Q\to \Pic(X(\fan))_\Q.
\end{align*}
We use the same symbol $f_A^*$ here to avoid inventing too many notations,
but it has the drawback of making confusions.
Thus we will state clearly whether we talk about divisors or divisor classes every time we use the symbol $f_A^*$.

What we do for pulling back $\Q$-Cartier divisors in the simplicial case is as follows. First,
notice that an element $D\in\CDivT(X(\fan'))_\Q$ can be identified with a {\em rational} support function
$\psi_D$, i.e., it takes rational values on the ray generators.
The composition $(\psi_D\circ A)$ is piecewise linear on the fan $\widetilde{\fan}$, but not on $\fan$.
We use the values of the function on rays to make an interpolation and obtain a piecewise linear function on $\fan$.
This step is possible because $\fan$ is simplicial.
If we denote the modifying (interpolation) function by $\mu=\mu_{\widetilde{\fan},\fan}$, we can
describe it more concretely.
Let $\varphi$ be a rational continuous piecewise linear function on $\widetilde{\fan}$,
for a maximal cone $\sigma\in\fan(n)$,
assume that $\tau_1\cdots\tau_n$ are one-dimensional faces of $\sigma$, with ray generators $v_1,\cdots,v_n$,
then $\mu(\varphi)|_{\sigma}=\sum_{i=1}^n \varphi(v_i)\cdot v_i^*$. Here $v_i^*$ is the dual basis of $v_i$
with respect to the basis $\{v_1,\cdots,v_n\}$.

To sum up, we have the following:

\begin{cor*}
For complete, simplicial toric varieties $X(\fan), X(\fan')$, and a dominant toric rational map
$f_A: X(\fan)\dashrightarrow X(\fan')$,
we can write the procedure of pulling back divisors as $f_A^* D = \mu_{\widetilde{\fan},\fan}(\psi_D\circ A)$.
\hfill\qed
\end{cor*}

%%%%%%%%%%%%%%%%%%%%%%%%%%%%%%%%%%%%%%%%%%%%%%%%%%%%%%%%%%%%%%%%%%%%%%
%%%%
%%%%  Section: Algebraic Stability
%%%%
%%%%%%%%%%%%%%%%%%%%%%%%%%%%%%%%%%%%%%%%%%%%%%%%%%%%%%%%%%%%%%%%%%%%%%

\section{Algebraic stability}
\label{sec:alg_stab}

For the rest of this paper, all toric varieties are assumed to be complete and simplicial.

\subsection{Definition and a geometric criterion}

\begin{defn}
A toric rational map $f_A:X(\fan)\dashrightarrow X(\fan)$ is {\em strongly algebraically stable} if
$(f_A^k)^*=(f_A^*)^k$ as maps of $\CDivT(X(\fan))_\Q$ for all $k\in\N$. It is {\em algebraically stable} if
$(f_A^k)^*=(f_A^*)^k$ as maps of $\Pic(X(\fan))_\Q$, for all $k$.
\end{defn}

Notice that $(f_A^k)^*=(f_A^*)^k$ on $\CDivT(X(\fan))_\Q$ implies $(f_A^k)^*=(f_A^*)^k$ on $\Pic(X(\fan))_\Q$,
so the condition for
strongly algebraic stability is indeed stronger than that for algebraic stability.
It is not clear to us whether the two conditions are equivalent or not
in general. However, if we assume that the toric variety $X=X(\fan)$ is projective, then
the two conditions are equivalent. We will prove that later in this section.

Our next goal is to prove a geometric characterization of strongly algebraically stable maps. We need to prove
a lemma first. Given two homomorphisms of lattices $A:N\to N'$ and $B: N''\to N$, they induce
two toric rational maps $f_A:X(\fan)\to f(\fan')$ and $f_B:X(\fan'')\to f(\fan)$.

\begin{lem}
$(f_A\circ f_B)^* = f_B^* \circ f_A^*$ as maps
\[
\CDivT(X(\fan'))_\Q\ \to \CDivT(X(\fan''))_\Q
\]
if and only if for each ray in $\fan''$, the cone closure of its image maps regularly to $\fan'$. That is, for each
$\tau\in\fan''(1)$, there exists a $\sigma'\in\fan'$ such that $A(\overline{B(\tau)})\subset \sigma'$.
\end{lem}

\begin{figure}
\centerline{  \includegraphics{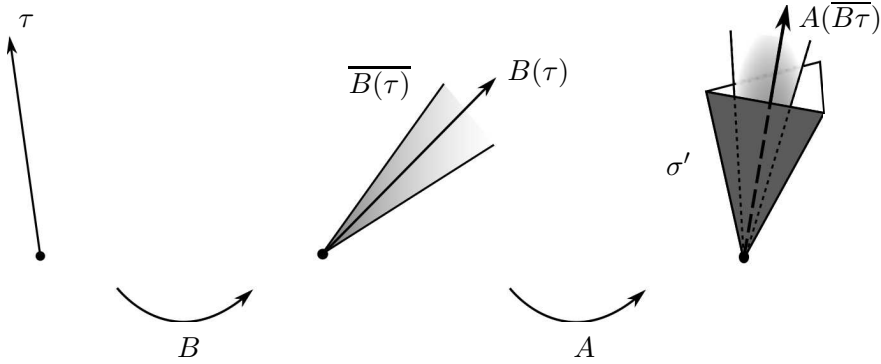} }
\begin{picture}(0,0)(0,0)
\put(35,105){\makebox{$B(\tau)$} }
\put(-25,100){\makebox{$\overline{B(\tau)}$} }
\put(145,125){\makebox{$A(\overline{B\tau})$} }
\put(-150,125){\makebox{$\tau$} }
\put(-90,0){\makebox{$B$} }
\put(60,0){\makebox{$A$} }
\put(95,69){\makebox{$\sigma'$} }
\end{picture}
\caption{The geometric condition for $(f_A\circ f_B)^* = f_B^* \circ f_A^*$.}
\end{figure}

\begin{proof}
First, suppose that the geometric condition is satisfied, we want to show $(f_A\circ f_B)^* = f_B^* \circ f_A^*$.
Remember that $(f_A\circ f_B)^* D = \mu (\psi_D\circ (A\circ B))$ and
$(f_B^* \circ f_A^*) D = f_B^*(f_A^* D) = \mu ( \mu(\psi_D\circ A)\circ B)$, where $\mu$
is the modifying function. So it is enough to show that,
for all $\tau_i\in\fan(1)$ and $v_i$ the ray generator of $\tau_i$,
\[
(\psi_D\circ (A\circ B))(v_i) =  (\mu(\psi_D\circ A)\circ B)(v_i),
\]
that is, $\psi_D(A( B v_i) ) =  \mu(\psi_D\circ A)( B v_i)$.

Since $A(\overline{B(\tau_i)})\subset \sigma'$ for some $\sigma'\in\fan'$ and $\psi_D$ is linear on $\sigma'$,
hence $(\psi_D\circ A)$ is linear on $\overline{B(\tau_i)}$. The interpolation $\mu$ therefore does not do
anything on $\overline{B(\tau_i)}$, and we have $\mu(\psi_D\circ A)( B v_i)= (\psi_D\circ A)( B v_i)$.

Conversely, if for some ray $\tau\in\fan(1)$, $\overline{B(\tau)}$ does not map regularly by $A$.
This means that $A(\overline{B(\tau)})$ is not contained in any cone of $\fan'$.
We will construct a divisor $D\in\CDivT(X(\fan'))$ such that $(f_A\circ f_B)^* D\ne (f_B^* \circ f_A^*)D$. Let
$\gamma_1,\cdots,\gamma_m$ be the one-dimensional faces of $\fan'$, and for $i=1,\cdots,m$, let
\[
a_i=\begin{cases} 0 &   \text{if $\gamma_i$ is a face of $\overline{(A\circ B)(\tau)}$,} \\
                  1 &   \text{otherwise.}  \end{cases}
\]
Define $D=\sum_{i=1}^m a_i\cdot V(\tau'_i)$, and let $\psi_D$ be the support function of $D$.
First, observe that $\psi_D(v)=0$ if and only if $w\in \overline{(A\circ B)(\tau)}$.
Thus for the divisor $(f_A\circ f_B)^* D$, the coefficient of $V(\tau)$ is $0$.

On the other hand, since $A(\overline{B(\tau)})$ is not contained in $\overline{(A\circ B)(\tau)}$,
there is some one dimensional face $\tau_0$ of $\overline{B(\tau)}$ such that
$A(\tau_0)\not\in \overline{(A\circ B)(\tau)}$.
Let $v_0$ be the ray generator of $\tau_0$, then $\psi_D(Av_0)>0$. Thus $\mu(\psi_D\circ A)$ is
strictly positive in the relative interior of $\overline{B(\tau)}$, which contains $Bv$. This implies
that the coefficient of $V(\tau)$ for the divisor $(f_B^* \circ f_A^*)D$ is strictly positive.
Therefore we have $(f_A\circ f_B)^* D\ne (f_B^* \circ f_A^*)D$.
\end{proof}

\begin{thm}\label{thm:sas}
A toric rational map $f_A:X(\fan)\dashrightarrow X(\fan)$ is strongly algebraically stable if and only if
for all ray $\tau\in\fan(1)$ and for all $n\in\N$, $\overline{A^n(\tau)}$ maps regularly to $\fan$ by $A$.
\end{thm}

\begin{proof}
First assume that $f=f_A$ is strongly algebraically stable. Thus for all $n\in\N$, we have $(f^n)^*=(f^*)^n$ and
$(f^{n+1})^*=(f^*)^{n+1}$. This gives us
\[
(f\circ f^n)^*=(f^{n+1})^*=(f^*)^{n+1}=(f^*)^n \circ f^* = (f^n)^*\circ f^*.
\]
By the above lemma, the equality $(f\circ f^n)^*=(f^n)^*\circ f^*$ implies that $\overline{A^n(\tau)}$
is mapped regularly to $\fan$ by $f$.

Conversely, assume that $\overline{A^n(\tau)}$ is mapped regularly to $\fan$ by $f$ for all $n\in\N$. This
tells us that $(f\circ f^n)^* = (f^n)^*\circ f^*$ for all $n\in N$. Thus we have $(f^n)^*=(f^*)^n$ for all $n\in\N$
by an induction argument.
\end{proof}

In fact, let $\sigma_n=\overline{A^n(\tau)}$, the next lemma implies that, not only $\sigma_n$ maps to
$\fan$ regularly, also the cone closure $\overline{A(\sigma_n)}$ is equal to
$\sigma_{n+1}=\overline{A^n(\tau)}$.

\begin{lem}\label{lem:coneclosure}
Assume further that $A_\R$ is surjective, and $\overline{B(\tau)}$ maps regularly to $\fan'$, then for all
$\tau\in\fan''$, we have $\overline{A(\overline{B(\tau)})}=\overline{(A\circ B)(\tau)}$.
\end{lem}

That is, if $\sigma$ is the smallest cone
in $\fan$ that contains $B(\tau)$, and $\sigma'$ is the smallest cone
in $\fan'$ that contains $A(\sigma)$, then $\sigma'$ will be the smallest cone in $\fan'$ that contains
$A(B(\tau))$.

\begin{proof}
Obviously, $\overline{(A\circ B)(\tau)}$ is a face of $\overline{A(\overline{B(\tau)})}$.
Thus there is a supporting hyperplane $H'$ of
$\overline{A(\overline{B(\tau)})}$ in $N'_\R$ such that
\[
\overline{A(\overline{B(\tau)})}\cap H'=\overline{(A\circ B)(\tau)}.
\]
The preimage $H=A^{-1}(H')$ will then be a
supporting hyperplane of $\overline{B(\tau)}$ in $N_\R$, so $\overline{B(\tau)}\cap H$ is a
face of $\overline{B(\tau)}$ that contains $B(\tau)$. By the minimality of $\overline{B(\tau)}$, we must have
$\overline{B(\tau)}\cap H= \overline{B(\tau)}$, i.e., $\overline{B(\tau)}\subset H$.
Thus, $A(\overline{B(\tau)})\subset H'$, and by the minimality of $\overline{A(\overline{B(\tau)})}$,
we know $\overline{A(\overline{B(\tau)})}\subset H'$. Therefore,
\[
\overline{A(\overline{B(\tau)})} = \overline{A(\overline{B(\tau)})}\cap H'=\overline{(A\circ B)(\tau)}.
\]
\end{proof}

With Theorem~\ref{thm:sas} and Lemma~\ref{lem:coneclosure}, we can describe the behavior, under iterations,
of an strongly algebraically stable toric rational map $f_A$ very concretely, as follows.
For each ray $\tau\in\fan(1)$, let $\sigma_1=\overline{A(\tau)}$ be the smallest cone containing $A(\tau)$,
then $\sigma_1$ will map regularly to some cone in $N$, Assume
$\sigma_2 = \overline{A(\sigma_1)} = \overline{A^2(\tau)}$ is the smallest such cone.
Here the second equality is due to the lemma. Then $\sigma_2$ will map regularly again to some
smallest $\sigma_3=\overline{A(\sigma_2)} = \overline{A^2(\sigma_1)}= \overline{A^3(\tau)}$, and so on.

\subsection{Algebraic stable vs. strongly algebraic stable}
Now we can prove the equivalence of algebraic stable and strongly algebraic stable in the projective case.
The equivalence of the two conditions, and a proof in the general case
is mentioned to us by C. Favre. We adapted his proof to a proof for toric varieties.

Given two integer matrices $A,B\in\M_n(\Z)$ with nonzero determinants, which
induce two dominant toric rational maps $f_A, f_B : X \dashrightarrow X$.

\begin{lem}
Let $D$ be an ample, $T$-invariant divisor on $X$, then the
difference $(f_B^*\circ f_A^*) D - (f_A \circ f_B)^* D$ is an effective $\Q$-Cartier divisor.
\end{lem}

\begin{proof}
Write
\[
(f_B^*\circ f_A^*)D = \sum_{\tau\in\fan(1)} a_\tau V(\tau),\ \ \
(f_A\circ f_B)^*D = \sum_{\tau\in\fan(1)} b_\tau V(\tau).
\]

We will show that $a_\tau\ge b_\tau$ for every $\tau\in\fan(1)$, which is equivalent to
the lemma.

Let $\psi=\psi_D$ be the support function of $D$. For some $\tau\in\fan(1)$, let $v\in\tau$
be the ray generator. Let $\sigma=\overline{B\tau}$ be the smallest cone which contains
$B\tau$, and assume that $u_1,\cdots,u_d$ are the generators of the cone $\sigma$. Then
there are positive numbers $r_1,\cdots,r_d$ such that $B(v) = r_1 u_1 + \cdots + r_d u_d$.

By the formula for pulling back divisors, to compute $a_\tau$, we need to apply the {\em interpolation}
process, and obtain
\[
a_\tau = -[r_1\psi(Au_1)+\cdots+r_d\psi(Au_d)].
\]
We can also see that
\[
b_\tau = -\psi((A\circ B)(v)) = -\psi ( r_1 Au_1 + \cdots + r_d Au_d ).
\]
Now the fact $a_\tau\ge b_\tau$ comes from the fact that $\psi$ is (strictly) convex
since $D$ is ample.
\end{proof}

\begin{prop}
For a projective, complete, simple toric variety $X = X(\fan)$,
 a toric rational map $f_A$ is strongly algebraically stable if and
only if it is algebraically stable.
\end{prop}

\begin{proof}
Since strongly AS implies AS, it suffices to show the other direction.
Assume that $f_A$ is not strongly AS, then there is a ray $\tau$ and
a positive integer $k$ such that $A^k(\overline{A\tau})$ is not contained
in any cone of $\fan$.

Let $D$ be any ample divisor, using the same notation as in the proof
of the above lemma, with $B = A^k$, we can see that $a_\tau > b_\tau$, since the
$A(u_i)$'s are not in a same cone, and $\psi$ is strictly convex.

Thus the difference between the support functions of $(f_A^{k+1})^*D$ and
that of $(f^*)^{k+1}D$ is a nonnegative function which is strictly positive
on $\tau$, hence cannot be linear.
This means $(f^{k+1})^*D \ne (f^*)^{k+1}D$ in $\Pic(X)$.
\end{proof}

\subsection{Applications of the criterion}
We will apply the above criterion (Theorem~\ref{thm:sas}) to give some results about stabilization in certain cases.

First, suppose all entries of $A$ are non-negative, i.e., $f_A$ is a polynomial monomial map. There is
a nice nonsingular toric model on which $f_A$ is algebraically stable, namely $(\P^1)^n$.

\begin{prop}
Every monomial polynomial map is strongly algebraically stable on $(\P^1)^n$, hence algebraically stable.
\end{prop}

\begin{proof}
Let $\fan$ be the fan such that $X(\fan)=(\P^1)^n$. The rays of $\fan$ are given by
$\tau_i=\R_{\ge 0}\cdot e_i$ and $-\tau_i$, for $i=1,\cdots,n$.
The morphism $A$ maps each of $\tau_i$ into the cone $\sigma_+$ generated by $e_1,\cdots,e_n$, and
maps each of $-\tau_i$ into the cone $\sigma_-$ generated by $-e_1,\cdots,-e_n$.

Observe that the compositions of polynomial maps are still polynomial maps.
So $A^k$ are all polynomial monomial maps for $k\ge 1$. Also notice that
$A^k(\tau_i)\subset\sigma_+$, so $\overline{A^k(\tau_i)}$ is a face of $\sigma_+$.
Hence there is a subset of indexes $I\subset \{1,\cdots,n\}$ such that $\overline{A^k(\tau_i)}$
is generated by $\{e_i|i\in I\}$. Since each $A^k(e_i)\in\sigma_+$, we have that
$A(\overline{A^k(\tau_i)})\subset\sigma_+$. This means $\overline{A^k(\tau_i)}$ maps regularly
for all $k$. By symmetry, we also know that $A(\overline{A^k(-\tau_i)})\subset\sigma_-$.
Therefore, the map $f_A$ is strongly algebraically stable on $X(\fan)=(\P^1)^n$.
\end{proof}

We will discuss more properties of monomial maps on $(\P^1)^n$ in Section~\ref{sec:P1n}.
The above property is about maps on a fixed toric variety $(\P^1)^n$. Next, we will fix some map, and
ask whether there exists a toric variety on which the map is strongly algebraically stable.
We give partial answers for maps satisfying some conditions.

\begin{thm}
\label{thm:stable}
Suppose that $A\in \M_n(\Z)$ is an integer matrix.
\begin{enumerate}
\item If there is a unique eigenvalue $\lambda$ of $A$ of maximal modulus, with algebraic multiplicity one;
      then $\lambda\in\R$, and there exists a simplicial toric birational model X (maybe singular)
and a $k\in\N$ such that $f_A^k$ is strongly algebraically stable on $X$.
\item If $\lambda,\bar{\lambda}$ are the only eigenvalues of $A$ of maximal modulus, also with algebraic
      multiplicity one, and if
      $\lambda= |\lambda|\cdot e^{2\pi\i\theta}$, with $\theta\not\in\Q$; then there is no toric birational model
      which makes $f_A$ strongly algebraically stable.
\end{enumerate}
\end{thm}

\begin{proof}
For (1), let $v\in \R^n$ be the eigenvector corresponding to the largest real eigenvalue $\lambda$, then the subspace
$\R v$ is attracting. We can find integral vectors $v_1,\cdots, v_n$, linearly independent over $\R$, such that
\begin{itemize}
\item $v$ is in the interior of the cone generated by $v_1,\cdots,v_n$.
\item $A^n (\R v_i)\to \R v$ for all $i=1,\cdots,n$, as elements of $\R\P^n$.
\end{itemize}
The rays $\{ \R_{\ge 0}\cdot v_i, \R_{\ge 0}\cdot (-v_i)\ |\ i=1\cdots,n\}$ generates a fan $\fan$
similar to the way we form $\P^1\times\cdots\times\P^1$. That is, the maximal cones of $\fan$ are
generated by the sets $\{s_1v_1,\cdots,s_nv_n\}$ where $s_i\in\{+1,-1\}$. All other cones are faces of
some maximal cone.
It is easy to see that for some $k$, $f_A^k$ is strongly algebraically stable on $X(\fan)$.

To prove (2), let $\lambda, \bar{\lambda}$ be the largest eigenvalue pair, and $\Gamma\subset\R^n$ be the
two dimensional invariant subspace corresponding to them. Since the fan $\fan$ is complete, there is at least
one ray $\tau\in\fan(1)$ such that under iterations, $A^k \tau$ will approach $\Gamma$. Moreover, since
$A|_\Gamma$ is an irrational rotation on rays, we know that for all $v\in\Gamma$, there is a sequence $k_i$ such
that $A^{k_i}\tau\to\R_{\ge 0}\cdot v$.

Consider the set $\fan\cap\Gamma=\{\sigma\cap\Gamma\ |\ \sigma\in\fan\}$, it is a fan in $\Gamma$. Each cone in it
is strictly
convex, but not necessarily rational. Pick $v_0 \in\Gamma$ which lies in the interior of some two dimensional cone
of $\fan\cap\Gamma$, and pick a sequence $k_i$ such that $A^{k_i}\tau\to\tau_0=\R_{\ge 0}\cdot v_0$.

Since $A^{k_i}\tau\to\tau_0$ and $\fan$ consists of only finitely many cones, there must be some $k$ such that
$\tau_0\in\overline{A^k\tau}$.
But $\tau_0$ is in the interior of some two dimensional cone of $\fan\cap\Gamma$, so we know that
$\overline{A^k\tau}\cap\Gamma$ is a two dimensional cone in $\Gamma$. Finally, we know that
$\overline{A^k\tau}\cap\Gamma$ cannot map regularly under all $A^k$, so $\overline{A^k\tau}$ cannot either.
Thus $A$ can never be made strongly algebraically stable.
\end{proof}

We do not know what the correct statement would be for
the missing case $\lambda= |\lambda|\cdot e^{2\pi\i\theta}$, with $\theta\in\Q$.

\begin{rem}
Some of our results in the last section and this section were obtained independently by Mattias Jonsson and
Elizabeth Wulcan ~\cite{JW}. They obtained the pull back formula and the criterion for stability.
One of the main theorems in their paper ~\cite[Theorem A']{JW} deal with smooth stabilization
of a monomial map by refining a given fan. This aspect of the stabilization is
more delicate and is not discussed in our paper.
Part (1) of Theorem~\ref{thm:stable} in the current paper is similar to the Theorem B in~\cite{JW}.
The difference is that they have further assumption on the eigenvalues, thus they can guarantee that $f_A$
is already stable. They also discuss the special case of monomial maps on toric surfaces (two dimensional
toric varieties), which is not dealt in this paper.
\end{rem}

%%%%%%%%%%%%%%%%%%%%%%%%%%%%%%%%%%%%%%%%%%%%%%%%%%%%%%%%%%%%%%%%%%%%%%
%%%%
%%%%  Section: Monomial Maps on Projective Spaces
%%%%
%%%%%%%%%%%%%%%%%%%%%%%%%%%%%%%%%%%%%%%%%%%%%%%%%%%%%%%%%%%%%%%%%%%%%%

\section{Monomial maps on projective spaces}

The motivation for studying toric rational maps comes from the study of monomial maps
on projective spaces. So let us come back to monomial maps and try to understand more
about them with the help of techniques
from toric varieties. Some results in this section is well known, but we give another
proof from a toric viewpoint.

\subsection{Pulling back divisors and divisor classes}
In this subsection, we will show that pulling-back divisors tells us information about homogenization of a monomial map
on projective spaces, and pulling-back divisors classes tells us information about the degree of a monomial map
on projective spaces.

Given an $n\times n$ integer matrix $A=(a_{i,j})_{1\le i, j\le n}$, the associated monomial map
$\C^n\to \C^n$ is given by
\[
\textstyle{(X_1,\cdots,X_n) \longmapsto\Bigl( \prod_{j=1}^n X_j^{a_{1,j}},\cdots ,\prod_{j=1}^n X^{a_{n,j}} \Bigr).  }
\]
Then we use the embedding $\C^n\hookrightarrow\P^n$ defined by $(X_1,\cdots, X_n)\mapsto [1;X_1;\cdots;X_n]$
to identify $\C^n$ with the open subset $U_0=\{x_0\ne 0\}\subset\P^n$. The inverse map $U_0\to\C^n$ is
given by $X_i=x_i/x_0$, and this is used to homogenize the monomial map.
After homogenizing, there is another integer matrix, with size $(n+1)\times (n+1)$, denoted by
$h(A)=(b_{i,j})_{0\le i, j\le n}$, such that
\[
f_A( [x_0;\cdots;x_n] )= \textstyle{\Bigl[ \prod_{j=0}^n x_j^{b_{0,j}}; \cdots ; \prod_{j=0}^n x_j^{b_{n,j}} \Bigr]}.
\]

Recall the structure of the fan associated to the projective space. The one dimensional cones are generated
by the standard basis
$e_1,\cdots, e_n$ and $e_0=-(e_1+\cdots +e_n)$. Denote them by $\tau_i=\R_{\ge 0}\cdot e_i$ for $i=0,\cdots,n$.
Consider the divisors $D_i=V(\tau_i)=\{x_i=0\}$. If we want to pull it back by $f_A$, what we do is
to pull back the defining equation.
This will give us the equation $\prod_{j=0}^n x_j^{b_{i,j}}=0$, which means
\begin{align*}
f_A^*(D_i) = b_{i,0}\cdot D_0 + b_{i,1}\cdot D_1 + \cdots+b_{i,n}\cdot D_n.
\end{align*}

On the other hand, by Theorem~\ref{thm:pullback},
if $\psi_i$ is the support function of the divisor $D_i$, then
\begin{align*}
f_A^*(D_i) = -\psi_i(Ae_0)\cdot D_0 -\psi_i(Ae_0)\cdot D_1 - \cdots - \psi_i(Ae_n)\cdot D_n.
\end{align*}
Thus we obtain the equality $b_{i,j}=-\psi_i(Ae_j)$.

\begin{ex}
\label{ex:unstable}
Consider the monomial map $f_A$ associated to the matrix
$A=\bigl( \begin{smallmatrix} -1& -2\\ 2&0 \end{smallmatrix} \bigr)$.
We know $f_A(X,Y)= (X^{-1}Y^{-2}, X^2)$ on $\C^2$, and then using homogenization, we can write down
the formula for $f_A:\P^2\dashrightarrow\P^2$ as $f_A[w;x;y]=[w^2xy^2 ; w^5 ; x^3y^2]$.

Let us consider the divisor $V(\tau_0)=\{w=0\}$. By pulling back the defining function, we know $f_A^*(V(\tau_0))$ is
defined by $w^2xy^2=0$, which, as a divisor, is $2\cdot V(\tau_0)+1\cdot V(\tau_1)+2\cdot V(\tau_2)$.
To apply the above discussion, remember that $V(\tau_0)$
corresponds to the support function $\psi_0$ on $\fan$ such that $\psi_0(e_0)=-1$ and $\psi_0(e_1)=\psi_0(e_2)=0$.
It is not hard to see
that $\psi_0(a_1,a_2)=\min\{0, a_1, a_2\}$. Therefore,
\begin{align*}
f_A^*(V(\tau_0)) &= -\psi_0(Ae_0)\cdot V(\tau_0) -\psi_0(Ae_1)\cdot V(\tau_1)-\psi_0(Ae_2)\cdot V(\tau_2)\\
             &= 2\cdot V(\tau_0)+1\cdot V(\tau_1)+2\cdot V(\tau_2).
\end{align*}
\end{ex}

In general, the formulae of $\psi_i$ for $\P^n$ is as follows.
\begin{equation}
\label{eq:supp_fun}
\left\{
\begin{split}
\psi_0 (a_1,\cdots,a_n) &= \min\{0,a_1,\cdots,a_n\},\\
\psi_i (a_1,\cdots,a_n) &= \min\{0, -a_i, a_j - a_i; j\ne i\}\text{ for $i=1,\cdots, n$.}
\end{split}
\right.
\end{equation}

Since the homogenization matrix $h(A)=(-\psi_i(Ae_j))$ is related to the pulling back divisors, we can translate
the condition of algebraic stability to a condition on $h(A)$.

\begin{prop}
A monomial map $f_A:\P^n\dashrightarrow\P^n$ is strongly algebraically stable
if and only if $h(A^k)=h(A)^k$ for $k=1,2,\cdots$.
\end{prop}

\begin{proof}
The entries of the $i$-th row in $h(A^k)$ are the coefficients of $(f_A^k)^*(V(\tau_i))$, whereas
the entries of the $i$-th row in $h(A)^k$ are the coefficients of $(f_A^*)^k(V(\tau_i))$. The proposition then follows
from the fact that $\CDivT(\P^n)$ is generated by $V(\tau_0),\cdots,V(\tau_n)$.
\end{proof}

What happened in unstable cases is that when we iterate the map $h(A)$ directly, some terms got canceled out.
For example, the map we mentioned above $f_A:[w;x;y]\mapsto[w^2xy^2 ; w^5 ; x^3y^2]$ is not stable,
because when we iterate once, the map
\[
[w;x;y]\mapsto[w^9x^8y^8;w^{10}x^5y^{10};w^{15}x^6y^4]
\]
which corresponds to $h(A)^2$, has a common factor $w^9x^5y^4$ on each component; so we need to
divide all components by $w^9x^5y^4$, and obtain $[w;x;y]\mapsto [x^3y^4;wy^6;w^6x]$, whose components
corresponds to $h(A^2)$.

Next, we turn our attention to the pull back of divisor classes, i.e., elements of $\Pic(\P^n)$.
It is well known that $\Pic(\P^n)\cong\Z$, and the isomorphism is given by the degree. We thus have the map
$\deg: \CDivT(\P^n)\to\Pic(\P^n)$
given by $\deg(\sum a_i V(\tau_i))=\sum a_i$. Furthermore, for a monomial map $f_A$ on the projective space,
it is easy to deduce that the pull back on Picard group is the same as the degree of the map, and is given by
$\deg(f_A^*D)$ for any divisor $D$ of degree one. We also denote this number by $\deg(f_A)$.
If $\psi$ is the support function for $D$, then we know the degree of the
monomial map $f_A$ is given by
\begin{equation}
\label{eq:deg}
\deg(f_A) = \sum_{i=0}^n -\psi(A e_i).
\end{equation}

For example, let $\psi$ be any one of the $\psi_i$ listed in (\ref{eq:supp_fun}),
then we can get a concrete formula for $\deg(f_A)$. In particular,
let $\psi=\psi_0$, we have
\begin{align*}
-\psi(a_1,\cdots,a_n) &= -\psi_0(a_1,\cdots,a_n) \\
                      &= -\min\{0,a_1,\cdots,a_n\}\\
                      &= \max\{0,-a_1,\cdots,-a_n\}.
\end{align*}
Then we rediscover the formula in~\cite[Proposition 2.14]{HP}.
\[
\deg(f_A) = \sum_{j=1}^n \max_{1\le i\le n}\{0,-a_{ij}\}+\max_{1\le i\le n}\Bigl\{0,\sum_{j=1}^n a_{ij}\Bigr\}.
\]

The definition of algebraic stability for rational maps on
$\P^n$ states that $f_A$ is algebraically stable if and only if
$\deg(f_A^k)=\deg(f_A)^k$ for all $k$.
%Each of the $V(\tau_i)$ is a divisor of degree one, so we can also compute $\deg(f_A)$ by summing up
%all numbers in any row of $h(A)$. If $h(A^k)\ne h(A)^k$, that means some factors are divided out in each of the
%term when we iterate the map, as shown in the above example, and we must have $\deg(f_A^k) < \deg(f_A)^k$.
%Therefore, taking the contrapositive, we obtain that algebraically stable implies strongly
%algebraically stable for monomial maps on $\P^n$. We conclude the discussion in the following.
%
%\begin{prop}
%For monomial maps on $\P^n$, algebraic stability and strongly algebraic stability are equivalent.
%\hfill\qed
%\end{prop}
Another property of the degree sequence is that it is {\em submultiplicative}, i.e.,
$\deg(f_A^{k+k'})\le \deg(f_A^k)\cdot\deg(f_A^{k'})$.
%This is implicitly implied in the above discussion.
We will use this property several times in the next subsection.

\subsection{Estimates of the degree sequence}
\label{sec:deg_seq_est}
In this subsection, we are going to study the degree sequence $\{\deg(f_A^k)\}_{k=1}^\infty$.
We are particularly interested in the asymptotic behavior of the degree sequence.
An important numerical invariant is the {\em asymptotic degree growth}
\[
\delta_1(f_A)=\lim_{k\to\infty} (\deg(f_A^k))^{\frac 1 k}.
\]
It is known that for a monomial map $f_A$, $\delta_1(f_A)=\rho(A)$, the spectral radius of the matrix $A$
(\cite[Theorem 6.2]{HP}).
We will refine this result and give more precise estimates on the degree growth of a monomial map.

\begin{rem}
For $\P^n$, the asymptotic degree growth is the same as the first dynamical degree,
which will introduce in more detail in Section ~\ref{sec:d1_mono}.
\end{rem}

For two sequences $\{\alpha_k\}_{k=1}^\infty$ and $\{\beta_k\}_{k=1}^\infty$ of positive real numbers,
we say that they are {\em asymptotically equivalent},
denoted by $\alpha_k\sim \beta_k$, if there exists two positive constants $c_1\ge c_0 > 0$,
independent of $k$, such that $c_0\cdot\beta_k \le \alpha_k \le c_1\cdot\beta_k$ for all $k$.

Let us start from some examples. In the following examples, we assume that $a>b>0$ are two natural numbers.

\begin{ex} For $A = \bigl(\begin{smallmatrix} a & 0 \\ 0 & b \end{smallmatrix}\bigr)$,
i.e., $f_A$ is the map
$(x,y)\mapsto (x^{a},y^{b})$ on $\C^2$. It is easy to see that $\deg(f_A^k)= a^k$.
\end{ex}

\begin{ex} For $A = \bigl(\begin{smallmatrix} a & 1 \\ 0 & a \end{smallmatrix}\bigr)$,
then $A^k = \bigl(\begin{smallmatrix} a^k & ka^{k-1} \\ 0 & a^k \end{smallmatrix}\bigr)$,
i.e., $f_A^k$ is the map
$(x,y)\mapsto (x^{a^k}y^{ka^{k-1}},y^{a^k})$ on $\C^2$. Thus for large $k$ (more precisely, for $k\ge a$),
\[
\deg(f_A^k)= a^k + ka^{k-1} = (1+k/a)\cdot a^k\sim k\cdot a^k.
\]
This shows that in the non-diagonalizable case, we may have some polynomial
multiplying with the power of the spectral radius in the estimate.
\end{ex}

\begin{ex}
\label{ex:C0=C1}
For $A = \bigl(\begin{smallmatrix} -a & 0 \\ 0 & b \end{smallmatrix}\bigr)$,
i.e., $f_A$ is the map
$(x,y)\mapsto (x^{-a},y^{b})$ on $\C^2$. It is not hard to verify that
\[
\deg(f_A^k)=\begin{cases} a^k & \text{ for $k$ even,} \\
                        a^k+b^k & \text{ for $k$ odd.}\end{cases}
\]
The degree depends on the parity of $k$, but we still have $\deg(f_A^k)\sim a^k$. Moreover, the
sequence $\{\deg(f_A^k)/a^k\}_{k=1}^\infty$ has only one limit point.
\end{ex}

\begin{ex}
\label{ex:C0not=C1}
For $A=\Bigl(\begin{smallmatrix} -a & 0 & 0 \\ 0 & -a & 0 \\ 0 & 0 & b \end{smallmatrix}\Bigr)$,
the map $f_A:\C^3\to\C^3$ is given by
$(x,y,z)\mapsto (x^{-a},y^{-a},z^b)$.
The degree sequence of $f_A$ is
\[
\deg(f_A^k)=\begin{cases} a^k & \text{ for $k$ even,}\\
                        2\cdot a^k+b^k & \text{ for $k$ odd.} \end{cases}
\]
We still have $\deg(f_A^k)\sim a^k$, but the
sequence $\{\deg(f_A^k)/a^k\}_{k=1}^\infty$ has two limit points: 1 and 2.
\end{ex}

\begin{ex}
Moreover, consider the monomial map $f:\C^n\dashrightarrow \C^n$ defined by
$f(x_1,\cdots,x_n)=(x_2^{-a},x_3^a,\cdots,x_n^a, x_1^a)$.
A careful calculation will show that $\deg(f^k)\sim a^k$, and
the sequence $\{\deg(f_A^k)/a^k\}_{k=1}^\infty$ has $n$ limit points: $1,2,\cdots,n-1$ and $n$.
\end{ex}

\begin{ex}
Consider the monomial map $f_A$ associated to the matrix
$A=\bigl( \begin{smallmatrix} -1& -2\\ 2&0 \end{smallmatrix} \bigr)$, as in Example~\ref{ex:unstable}.
The matrix $A$ has two conjugate eigenvalues $\lambda = (-1+ \sqrt{-15})/2$ and $\bar{\lambda}$; and
$\lambda/\bar{\lambda}$ is not a root of unity. We will show in Theorem~\ref{thm:deg_bounds} that
$\deg(f_A^k)\sim |\lambda|^k$, but when we consider the sequence $\{\deg(f_A^k)/|\lambda|^k\}_{k=1}^\infty$,
we no longer have finitely many limit points. In fact, we will prove that the sequence is dense in some
interval (Proposition~\ref{prop:irrational}).
\end{ex}

The main result for general monomial maps is the following theorem.

\begin{thm}
\label{thm:deg_bounds}
Given an $n\times n$ integer matrix $A$ with nonzero determinant, assume that $\rho(A)$ is the
spectral radius of $A$.
Then there exist two positive constants $C_1\ge C_0>0$
and a unique integer $\ell$ with $0\le\ell\le n-1$, such that
\begin{equation}
C_0\cdot k^\ell\cdot\rho(A)^k \le \deg(f_A^k) \le C_1\cdot k^\ell\cdot\rho(A)^k
\end{equation}
for all $k\in\N$. Or, equivalently, $\deg(f_A^k) \sim k^\ell\cdot\rho(A)^k$.

In fact, $(\ell+1)$ is the size of the largest Jordan block of $A$ among the ones corresponding
to eigenvalues of maximal modulus.
\end{thm}

The idea we use to prove the theorem is the following observation. The assignment
$A\mapsto \deg(f_A)$ can be extended naturally to a function $\nu:\M_n(\R)\to \R$, and
the function $\nu$ is {\em almost} a norm on $\M_n(\R)$. Thus some
techniques on norms also applies to the study of degrees.

More precisely, in formula (\ref{eq:deg}), notice that the right hand side
can be defined over the real numbers because $\psi$ is a continuous piecewise linear function defined on
$N_\R\cong \R^n$. The only requirement is that the associated divisor of $\psi$ has degree one.
Thus, we define a function $\nu:\M_n(\R)\to \R$ by
\begin{equation}
\label{eq:nu}
\nu(M) = \sum_{i=0}^n -\psi(Me_i).
\end{equation}

\begin{prop}
The following properties hold for the function $\nu$.
\begin{itemize}
\item[(i)] Any support function $\psi$ of a $T$-divisor of degree one on $\P^n$ will give the same $\nu$, i.e.,
           $\nu$ is independent of the choice of $\psi$.
\item[(ii)] $\nu$ is a continuous function when we equip $\M_n(\R)\cong \R^{n^2}$ and $\R$ with the usual topology
      of the Euclidean spaces.
\item[(iii)] $\nu(M)\ge 0$, and $\nu(M)=0$ if and only if $M=0$.
         Thus, in fact, we have $\nu:\M_n(\R)\to \R_{\ge 0}$.
\item[(iv)] $\nu(rM)=r\cdot\nu(M)$ for $r\ge 0$.
\item[(v)] $\nu(M+M')\le \nu(M) + \nu(M')$.
\end{itemize}
\end{prop}

\begin{proof}
First, notice that (ii) is true because $\psi$ is continuous,
and (iv) is true because $\psi$ is linear on each ray.
Then (i) follows by (ii), (iv), and the fact that $\nu(A)=\deg(f_A)$ for $A\in\M_n(\Z)$, which is independent of $\psi$.

Once we know that $\nu$ is independent of the choice of $\psi$, one can pick any $\psi$, e.g. $\psi=\psi_0$,
and prove (iii) and (v) directly. However, we would like to offer a more intrinsic explanation for
(iii) and (v).

Since $\psi$ is the support function for a degree one divisor $D$ on $\P^n$, we know that $D$ is
very ample, and hence $\psi$ is strictly convex (see~\cite[p.70]{Fu}).
The first part of (iii), and (v), can be easily deduced
from convexity. Strict convexity is needed to show that $\nu(M)=0$ implies $M=0$.

Suppose $M\ne 0$, then $Me_0,Me_1,\cdots,Me_n$ cannot be all zero. But since $Me_0+\cdots+Me_n=0$,
and the cones in the fan for $\P^n$ are strongly convex (they do not contain any line through the origin),
$Me_0,\cdots, Me_n$ cannot all lie in the same cone. Thus by strict convexity, we know
\[
\nu(M)= - \sum_{i=0}^n \psi(Me_i)> -\psi(\sum_{i=0}^n Me_i)=\psi(0)=0.
\]
\end{proof}

By properties (iii)--(v), we know that the only reason to prevent $\nu$ from
being a norm is that we may have $\nu(M)\ne\nu(-M)$. Indeed, for the $n\times n$ identity
matrix $I_n$, we have $\nu(I_n)=1$, but $\nu(-I_n)=n-1$. So $\nu$ is not a norm.
However, if we define $\bar{\nu}(M)=\nu(M)+\nu(-M)$, then $\bar{\nu}$ is a norm.

Before we prove Theorem~\ref{thm:deg_bounds}, we need an elementary lemma from linear algebra.
\begin{lem}
\label{lemma:norm}
For an $n\times n$ matrix $A\in \M_n(\C)$ and any norm $\norm{\cdot}$ defined on $\M_n(\C)$, there
exists two positive constants $c_1\ge c_0>0$
and a unique integer $\ell$ with $0\le\ell\le n-1$, such that
\begin{equation}
c_0\cdot k^\ell\cdot\rho(A)^k \le \norm{A^k} \le c_1\cdot k^\ell\cdot\rho(A)^k
\end{equation}
for all $k\in\N$. Here $\rho(A)$ is the spectral radius of $A$, and $(\ell+1)$ is the size
of the largest Jordan block among those blocks corresponding
to eigenvalues of maximal modulus $\rho(A)$.
\end{lem}

\begin{proof}
It suffices to prove the lemma for the $L^\infty$ norm on $\M_n(\C)$. Thus, for $A=(a_{ij})$, we set
$\norm{A}=\norm{A}_\infty=\max_{i,j}\{|a_{ij}|\}$ for the rest of the proof.

Observe that $\norm{AB}\le n\cdot\norm{A}\cdot\norm{B}$. If we write $A=PJP^{-1}$, where $J$ is the Jordan canonical
form of $A$, then we have
\[
(n^2\cdot\norm{P}\cdot\norm{P^{-1}})^{-1}\cdot\norm{J^k}\le \norm{A^k}\le (n^2\cdot\norm{P}\cdot\norm{P^{-1}})\cdot\norm{J^k}.
\]
For large $k$, it is easy to see that $\norm{J^k}=\binom{k}{\ell}\cdot\rho(A)^k\sim k^\ell\cdot\rho(A)^k$, where
$(\ell+1)$ is as described above. Hence the lemma follows.
\end{proof}

Now we are ready to proof the theorem.

\begin{proof}[Proof of Theorem~\ref{thm:deg_bounds}]
For the matrix $A\in\M_n(\Z)$, consider the set
\[
\Bigl\{ \frac{A^k}{k^\ell\rho(A)^k} ~\Bigl|~ k\in\N\Bigr\}\subset \M_n(\R).
\]
By lemma~\ref{lemma:norm},
it is a subset of a compact set $S=\{ M\in\M_n(\R) ~|~ c_0\le \norm{M}\le c_1\}$ for some $c_1\ge c_0>0$.
Since $\nu$ is continuous, we have $\nu(S)\subset [C_0,C_1]$ for some reals $C_1\ge C_0\ge 0$. Moreover,
$0\not\in S$, thus $C_0>0$. This gives us
\[
C_0 \le \nu\Bigl(\frac{A^k}{k^\ell\cdot\rho(A)^k}\Bigr) \le C_1.
\]
for all $k\in\N$, with $C_1\ge C_0>0$.

Finally, since $k^\ell\cdot\rho(A)^k >0$, and $\nu(A^k)=\deg(f_A^k)$, we have
\[
C_0\cdot k^\ell\cdot\rho(A)^k \le \deg(f_A^k) \le C_1\cdot k^\ell\cdot\rho(A)^k
\]
This concludes the proof.
\end{proof}

\begin{cor}
\label{cor:deg_bounds}
If $A$ is diagonalizable, then we have
\begin{equation}
\label{eq:diag}
C_0\cdot\rho(A)^k \le \deg(f_A^k) \le C_1\cdot\rho(A)^k
\end{equation}
for some constants $C_1\ge C_0\ge 1$.
\end{cor}

\begin{proof}
In the diagonalizable case, $\ell=0$, hence we have (\ref{eq:diag}).
Recall that the degree sequence is submultiplicative. Thus, if we have $\frac{\deg(f_A^k)}{\rho(A)^k}= r <1$ for some $k$,
then
\[
\frac{\deg(f_A^{kj})}{\rho(A)^{kj}}\le \frac{\deg(f_A^k)^j}{\rho(A)^{kj}}= r^j \to 0 \text{ as $j\to +\infty$.}
\]
This contradicts the existence of $C_0>0$. Therefore, $\frac{\deg(f_A^k)}{\rho(A)^k}\ge 1$ for all $k$, and we can choose $C_0\ge 1$.
\end{proof}

If we impose more conditions on the matrix $A$, we can obtain more precise estimates on the degree sequence.

\begin{thm}
\label{thm:degseq1}
Assuming that the matrix $A$ is diagonalizable,
and there is a unique eigenvalue $\lambda_1$ of maximal modulus, which is real and positive.
Also, assume that the eigenvalues of $A$
are arranged as $\lambda_1 >|\lambda_2|\ge |\lambda_3|\ge \cdots\ge|\lambda_m|$ for some $~m$.
Then there is a constant $C\ge 1$ such that
\[
\deg(f_A^k)=C\cdot \lambda_1^k + O(|\lambda_2|^k).
\]
\end{thm}

\begin{proof}
First, given a vector $v\in\R^n$, since $A$ is diagonalizable, we can represent $v$ uniquely as
\[
v = v_1 + v_2 + \cdots + v_m,
\]
where each $v_j\in\C^n$ is an eigenvector corresponding to $\lambda_j$. We have $v_1\in\R^n$ since $\lambda_1$ is real.
Thus
\[
A^k v = \lambda_1^k v_1 + \lambda_2^k v_{2}+\cdots+\lambda_m^k v_{m}.
\]
Let $\psi$ be the support function of some degree one divisor $D$ in $\P^n$.
For each $k$, there is some maximal cone $\sigma_{k}$ such that $A^k v\in\sigma_{k}$. Let
$L_{k}$ be the linear function such that $L_{k}|_{\sigma_{k}}=\psi|_{\sigma_{k}}$.
Notice that $L_k$ can be defined on $\C^n$ as a linear map, and we have
\begin{align*}
\psi(A^k v) &= L_{k}(A^k v) = L_{k}\bigl(\sum_{j=1}^m \lambda_j^k v_{j}\bigr)\\
         &= \lambda_1^k \cdot L_{k}(v_{1}) + \sum_{j=2}^m \lambda_j^k \cdot L_{k}(v_{j})\\
         &=  L_{k}(v_1)\cdot\lambda_1^k+O(|\lambda_2|^k).
\end{align*}
There are two cases here: $v_1\ne 0$, or $v_1=0$.

First, if $v_1\ne 0$, then for the
rays $\tau = \R_{\ge 0}\cdot v$, we know $A^k\tau\to \R_{\ge 0}\cdot v_{1}$.
Thus for large $k$, we can choose $\sigma_k$ so that both $A^k v\in\sigma_k$ and $v_1\in\sigma_k$.
Since $L_{k}|_{\sigma_{k}}=\psi|_{\sigma_{k}}$ for the cone $\sigma_{k}$,
we know that for large $k$, the value $L_{k}(v_{1})=\psi(v_{1})$ is independent of $k$, and
$\psi(A^k v) = \psi(v_{1})\cdot\lambda_1^k+O(|\lambda_2|^k)$.
Second, if $v_1=0$, then it is obvious that $\psi(A^k v) = O(|\lambda_2|^k)$.

Now let's look at the fan structure of projective spaces.
For the ray generators $e_0,e_1,\cdots,e_n$ of $\P^n$, if $e_i$ is decomposed as
\begin{equation}
\label{eq:decomp}
e_i = v_{i,1} + v_{i,2} + \cdots + v_{i,m}
\end{equation}
for $i=0,1,\cdots,n$, where each $v_{i,j}$ is an eigenvector corresponding to the eigenvalue $\lambda_j$.
Then
\[
\psi(A^k e_i) = \psi(v_{i,1})\cdot\lambda_1^k+O(|\lambda_2|^k).
\]
If we set
\begin{equation}
\label{eq:const}
C=\sum_{i=0}^n -\psi(v_{i,1}),
\end{equation}
then we can compute the degree sequence $\deg(f_A^k)$ as
\begin{align*}
\deg(f_A^k) &= \deg((f_A^k)^* D) \\
            &= \sum_{i=1}^n -\psi(A^k e_i)\\
            &= C\cdot \lambda_1^k + O(|\lambda_2|^k).
\end{align*}
The fact that $C\ge 1$ is a consequence of Corollary~\ref{cor:deg_bounds}.
\end{proof}

Notice that, on our way to prove the theorem, we also derive a concrete formula for the constant $C$ in (\ref{eq:const}).

\begin{thm}
\label{thm:degseq2}
Assuming that the matrix $A$ is diagonalizable,
and there is a unique eigenvalue $\lambda_1$ of maximal modulus, which is real and negative.
Also assume that the eigenvalues of $A$
are arranged as $(-\lambda_1)>|\lambda_2|\ge |\lambda_3|\ge \cdots\ge|\lambda_m|$ for some $~m$.
Then there are two positive constants $C_0,C_1$, not necessarily distinct,
and satisfying $1\le C_0\le C_1^2$, such that
\[
\deg(f_A^{2k+l})=C_l\cdot |\lambda_1|^{2k+l} + O(|\lambda_2|^{2k+l}),
\]
where $l=0,1$.
\end{thm}

\begin{proof}
We consider the subsequences $\{\deg(f_A^{2k})\}$ and $\{\deg(f_A^{2k+1})\}$.
Since $A^2$ satisfies the condition in Theorem~\ref{thm:degseq1}, with the unique eigenvalue
$|\lambda_1|^2$ with maximal modulus, thus
\[
\deg(f_A^{2k})=C_0\cdot |\lambda_1|^{2k} + O(|\lambda_2|^{2k})
\]
for some $C_0\ge 1$. For the subsequence $\{\deg(f_A^{2k+1})\}$, we consider $Ae_i$ instead of $e_i$
in (\ref{eq:decomp}) in the proof of Theorem~\ref{thm:degseq1}, and apply the map $f_A^2$ on these vectors.
We then get
\[
\deg(f_A^{2k+1})=C_1\cdot |\lambda_1|^{2k+1} + O(|\lambda_2|^{2k+1})
\]
for some $C_1\ge 1$.

Finally, for any $k$, we have,
\[
\deg(f_A^{4k+2})/|\lambda_1|^{4k+2}\le \bigl(\deg(f_A^{2k+1})/|\lambda_1|^{2k+1}\bigr)^2.
\]
As $k\to\infty$, the left side converges to $C_0$, while the right side converges to $C_1^2$.
So the relation $C_0\le C_1^2$ follows. This completes the proof.
\end{proof}

From the proof, we cannot tell if the two constants $C_0$ and $C_1$ are the same or not.
Notice that Examples ~\ref{ex:C0=C1} and ~\ref{ex:C0not=C1} are both examples of the theorem.
However, we have $C_0=C_1=1$ for Example~\ref{ex:C0=C1}, but $C_0=1$, $C_1=2$
for Example~\ref{ex:C0not=C1}.
This shows that both cases are possible.

The idea in the proof of Theorem~\ref{thm:degseq2} of considering subsequences can be pushed further
to prove the following more general result.

\begin{thm}
\label{thm:periodic_constants}
Assume that the matrix $A$ is diagonalizable,
and assume for each eigenvalue $\lambda$ of $A$ of maximum modulus, $\lambda/\bar{\lambda}$ is a root of unity.
Then there is a positive integer $p$, and $p$ constants $C_0,C_1,\cdots, C_{p-1}\ge 1$, such that
\[
\deg(f_A^{pk+l})=C_l\cdot |\lambda_1|^{pk+l} + O(|\lambda_2|^{pk+l}),
\]
where $l=0,1,\cdots,p-1$.
\end{thm}

\begin{proof}
Notice that there is an integer $p$ such that the eigenvalue of $A^p$ of maximum modulus is unique
and positive, so we can use the same argument as Theorem~\ref{thm:degseq2} to the subsequences
\[
\{\deg(f_A^{pk})\},\{\deg(f_A^{pk+1})\}, \cdots, \{\deg(f_A^{pk+p-1})\}.
\]
The theorem then follows.
\end{proof}

Under the assumption of Theorem~\ref{thm:periodic_constants},
the sequence $\{\deg(f_A^{k})/|\lambda_1|^{k}\}_{k=1}^\infty$
has finitely many limit points, namely, $C_0,\cdots, C_{p-1}$.
The following proposition shows a different behavior of the sequence $\{\deg(f_A^{k})/|\lambda_1|^{k}\}_{k=1}^\infty$
when we have a maximal eigenvalue $\lambda$
such that $\lambda/\bar{\lambda}$ is {\em not} a root of unity.
Therefore, we cannot expect Theorem~\ref{thm:periodic_constants} holds for general diagonalizable matrices.

\begin{prop}
\label{prop:irrational}
For a $2\times 2$ integer matrix $A$, suppose it has a conjugate pair $\lambda,\bar{\lambda}$ of eigenvalues
such that $\lambda/\bar{\lambda}$ is not a root of unity. Then the sequence $\{\deg(f_A^{k})/|\lambda|^{k}\}_{k=1}^\infty$
is dense in some closed interval contained in $[1,\infty)$.
\end{prop}

\begin{proof}
First, notice that
\begin{equation}
\label{eq:normalized_degree}
\frac{\deg(f_A^{k})}{|\lambda|^{k}}=\frac{\nu(A^k)}{|\lambda|^{k}}=\nu\Bigl(\frac {A^k} {|\lambda|^{k}}\Bigr)
=\nu\Bigl( ({A} / {|\lambda|})^k\Bigr).
\end{equation}
Since $\lambda/\bar{\lambda}$ is not a root of unity, we can conjugate $A/{|\lambda|}$ to some irrational rotation matrix,
i.e., we can write
\[
A/{|\lambda|}=P \cdot\begin{pmatrix} \cos\theta & -\sin\theta \\ \sin\theta & \cos\theta \end{pmatrix}\cdot P^{-1},
\]
for some $\theta\not\in 2\pi\Q$. Thus the closure of the set $S=\{({A} / {|\lambda|})^k|k\in\N\}$ is
\[
\overline{S}=\Bigl\{ P \cdot
   \bigl(\begin{smallmatrix} \cos t & -\sin t \\ \sin t & \cos t \end{smallmatrix}\bigr)\cdot P^{-1}
   ~\Bigl|~ t\in[0,2\pi]~\Bigr\}.
\]

$\overline{S}$ is, topologically, a circle inside $\M_2(\R)$.
Since $\nu$ is continuous, $\nu(\bar{S})=\overline{\nu(S)}$ is connected and compact.
Thus it is either a point or a closed interval.

We claim that $\overline{\nu(S)}$ cannot be a point.
If $\overline{\nu(S)}=\{C\}$, then we will have $\deg(f_A^{k})=C\cdot |\lambda|^{k}$
for all $k\in\N$.
In this case, the degree sequence $d_k=\deg(f_A^k)$ satisfies a linear recurrence $d_{k+1}=|\lambda|\cdot d_k$.
This contradicts a theorem of Bedford and Kim~\cite[Theorem 1.1]{BK}, which asserts that if
the matrix $A$ has a complex eigenvalue $\lambda$ of maximal modulus, and $\lambda/\bar{\lambda}$ is not
a root of unity, then the degree sequence for $f_A$ cannot satisfy any linear recurrence relation.

Hence, $\nu(\bar{S})=\overline{\nu(S)}$ is a closed interval. By (\ref{eq:normalized_degree}), $\nu(S)$ is exactly the set
$\{\deg(f_A^{k})/|\lambda|^{k};k\in\N\}$.
Finally, by Corollary~\ref{cor:deg_bounds}, we further know that the interval $\overline{\nu(S)}$ is contained in $[1,+\infty)$.
This concludes the proof.
\end{proof}

\subsection{Degree growth on weighted projective spaces}
\label{subsec:WPS}

Weighted projective spaces are generalizations of the usual projective spaces.
The results we obtained in the last
subsection about the degree growth of monomial maps on projective spaces can be generalized to weighted projective spaces.
We will explain briefly how the generalization is done in this section.

For arbitrary positive integers $d_0,\cdots, d_n$, the associated {\em weighted projective space},
denoted by $\P(d_0,\cdots,d_n)$, is defined as
\[
\P(d_0,\cdots,d_n)= (\C^{n+1}-\{0\})/\sim
\]
where the equivalent relation is given by $(x_0,\cdots,x_n)\sim (\zeta^{d_0} x_0,\cdots, \zeta^{d_n} x_n)$ for $\zeta\in\C^*$.

For $d=\gcd(d_0,\cdots,d_n)$, one can show that $\P(d_0,\cdots,d_n)\cong \P(d_0/d,\cdots,d_n/d)$ (\cite[Proposition 3.6 (I)]{Reid}).
Moreover, suppose that $d_0,\cdots,d_n$  have no common factor, and that $d$ is a common
factor of all $d_i$ for $i \ne j$ (and therefore $d$ is coprime to $d_j$). Then
\[
\P(d_0,\cdots,d_n)\cong \P\Bigl(\frac {d_0} d , \cdots , \frac {d_{j-1}} d , d_j , \frac {d_{j+1}} d , \cdots, \frac {d_n} d \Bigr),
\]
(see ~\cite[Proposition 3.6 (II)]{Reid}).
A weighted projective space $\P(d_0,\cdots,d_n)$ such
that no $n$ of the $d_0, d_1,\cdots,d_n$ have a common factor is called {\em well formed}. The above isomorphism allows us
only consider the weighted projective spaces which are well formed.
We will make that assumption from now on. Also, for simplicity of notation, we
will denote $\P(d_0,\cdots,d_n)$ simply by $\P$ when there is no confusion. The usual projective space, which is
$\P(1,1,\cdots,1)$, will still be denoted by $\P^n$.

To construct $\P(d_0,\cdots,d_n)$ as a toric variety, one uses the same fan as in the construction of the
projective spaces. That is, the cones are generated by proper subsets of $\{e_0,\cdots,e_n\}$. The lattice $N'$ is taken to
be generated by the vectors $e_i':= e_i/ d_i $, $i=0,\cdots, n$.
Let $\tau_i = \R_{\ge 0}\cdot e_i$ be the rays for the fan of $\P$, the well formed-ness of $\P$ implies that $e_i'$ is
the ray generator for $\tau_i$ for $i=0,\cdots,n$.

\begin{ex}
\label{ex:P123}
The lattice and fan structure of $\P(1,2,3)$ is shown in Figure~\ref{fig:P123}. The solid black dots represent the lattice
$N=\Z^2$. The unfilled dots, together with the solid dots, represent the lattice $N'=\frac 1 2\Z\oplus\frac 1 3\Z$ for
$\P(1,2,3)$. The fan structure of $\P(1,2,3)$ is the same as $\P^2$.
\begin{figure}
\label{fig:P123}
\centerline{  \includegraphics{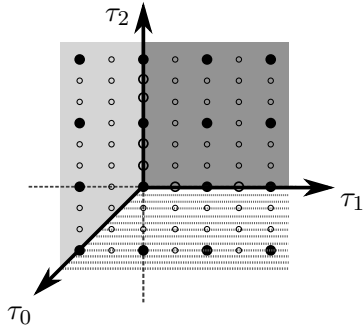} }
\begin{picture}(0,0)(0,0)
\put(-65,7){\makebox{$\tau_0$} }
\put(61,50){\makebox{$\tau_1$} }
\put(-28,120){\makebox{$\tau_2$} }
\end{picture}
\caption{The lattice and fan structure for $\P=\P(1,2,3)$.}
\end{figure}
\end{ex}

If we define the map $\theta: \Z^{n+1}\to N'$ by $\theta(a_0,\cdots,a_n)=a_0 e_0' +\cdots+ a_n e_n' $, then $\theta$ is a surjective
homomorphism, and $\ker(\theta)$ is the rank one subgroup $\Z\cdot(d_0,\cdots,d_n)$. Hence $N'\cong \Z^{n+1}/ \Z\cdot(d_0,\cdots,d_n)$,
and we also obtain a description for the dual lattice as follows:
\[
M'=(N')^\vee \cong \bigl\{ (a_0,\cdots,a_n)\in\Z^{n+1} ~|~ a_0d_0+\cdots+a_nd_n=0 \bigr\}.
\]

Recall that the group of $T$-invariant Weil
divisors is freely generated by the orbit closures $V(\tau_i)$, i.e., $\WDivT(\P)\cong\oplus_{i=0}^n \Z\cdot V(\tau_i)$.
Define the {\em weighted degree homomorphism} $\deg': \WDivT(\P) \to \Z$ by
$\deg'(a_i \cdot V(\tau_i) ) = \sum_{i=0}^n a_i d_i$.
The map is surjective since $\gcd(d_0,\cdots,d_n)=1$. Moreover, it is easy to see that the kernel
is canonically isomorphic to $M'$. Therefore, the divisor class group $A_{n-1}(\P)\cong \Z$, and the
isomorphism is induced by the weighted degree.

Let $m=\lcm(d_0,\cdots,d_n)$, one can show that a $T_N$-invariant Weil divisor $D$ is Cartier if and only if $m | \deg(D)$.
As a consequence, the image of the Picard group $\Pic(\P)\subset A_{n-1}(\P)\cong \Z$ under the  isomorphism is the subgroup $m\Z$.

Therefore, after tensoring with the group of rational numbers $\Q$, the group of $T$-invariant $\Q$-Weil divisors is
the same as that of $\Q$-Cartier divisors,
and the divisor class group is the same as the Picard group, both with $\Q$-coefficients.
Thus we will look at $\Q$-Weil divisors and rational support function in this subsection.

Let $\psi$ be a rational support function, then $\psi$ induces a $\Q$-Cartier divisor on $\P$, whose associated
$\Q$-Weil divisor is $D'=\sum_{i=0}^n -\psi(e'_i)\cdot V(\tau_i)$.
Also, $\psi$ induces
a $\Q$-Cartier divisor on $\P^n$, with associated
$\Q$-Weil divisor $D=\sum_{i=0}^n -\psi(e_i)\cdot V(\tau_i)$.
A basic fact is the following.

\begin{lem}
\label{lem:deg_WPS}
Assume the above notations, then the weighted degree of $D'$ is the same as the degree of $D$, i.e., $\deg'(D')=\deg(D)$.
\end{lem}

\begin{proof}
This can be verified as follows:
\[
\deg'(D')=\sum_{i=0}^n -d_i\cdot\psi(e'_i) =  \sum_{i=0}^n -d_i\cdot\psi(e_i/d_i) = \sum_{i=0}^n -\psi(e_i)=\deg(D).
\]
\end{proof}

We can then discuss the pull back
map for a toric map on a weighted projective space. The pull back map on divisors can be obtained by the formula
in Theorem~\ref{thm:pullback}. The pull back map on Picard group is given by the action on the degree of a divisor,
which we also call it the {\em weighted degree} of the map, denoted by $\deg'(f_A)$.
In general, the weighted degree is a rational number, not necessarily an integer.

More precisely, let $A\in\End(N')$, then $A$ induces a toric rational map $f_A:\P\to\P$.
Using the standard basis $e_1,\cdots,e_n$ of $N'_\R\cong N_\R$, we can represent $A$ as an $n\times n$ matrix
with {\em rational} entries.
The following proposition tells us how to compute the weighted degree of $f_A$.

\begin{prop}
Assume the above notations, then the weighted degree of $f_A$ is given by
\[
\deg'(f_A) = \nu (A),
\]
where $\nu:\M_n(\R)\to\R$ is the function defined  in (\ref{eq:nu}).
\end{prop}

\begin{proof}
The weighted degree can be computed
as $\deg'(f_A) = \deg(f_A^* D')$ for any $\Q$-divisor $D'$ on $\P$ of degree one. Thus, if $D'$ is a $\Q$-divisor on $\P$
of degree one, and $\psi=\psi_{D'}$ is the $\Q$-support function of $D'$, then
\[
\deg'(f_A) = \sum_{i=0}^n -d_i\cdot\psi(Ae'_i)=  \sum_{i=0}^n -\psi(Ae_i) = \nu (A).
\]
The last equality holds because the degree of the $\Q$-divisor on $\P^n$ associated to $\psi$ also has degree
one by Lemma~\ref{lem:deg_WPS}.
\end{proof}

\begin{ex}
For the matrix $A = \Bigl(\begin{smallmatrix} 1 & -\frac 3 2 \\  \frac 2 3 & 0 \end{smallmatrix}\Bigr)$, it does not preserve
the lattice $\Z^2$, but it preserve the lattice $N'=\frac 1 2\Z\oplus\frac 1 3\Z$ for $\P(1,2,3)$ in Example~\ref{ex:P123}.
In fact, we have $e'_1\mapsto e'_1+e'_2$ and $e'_2\mapsto -e'_1$. Therefore, $f_A$ is a well-defined toric rational map on
$\P(1,2,3)$. The map $f_A$ has degree $\frac {13}6$, which is not an integer.
\end{ex}

Since the weighted degree function is the same as the function $\nu$,
this tells us that the weighted degree growth of iterations of toric
rational maps on $\P$ follows the same results as the degree growth on $\P^n$.
Therefore, we can conclude that all theorems and propositions in Section~\ref{sec:deg_seq_est} also hold for
weighted degree growth of toric rational maps on weighted projective spaces.

%%%%%%%%%%%%%%%%%%%%%%%%%%%%%%%%%%%%%%%%%%%%%%%%%%%%%%%%%%%%%%%%%%%%%%
%%%%
%%%%  Section: Monomial Maps and polynomial maps on $(\P^1)^n$
%%%%
%%%%%%%%%%%%%%%%%%%%%%%%%%%%%%%%%%%%%%%%%%%%%%%%%%%%%%%%%%%%%%%%%%%%%%

\section{Monomial maps and polynomial maps on $(\P^1)^n$}
\label{sec:P1n}
In this section, we fix the space $(\P^1)^n=\underbrace{\P^1\times\cdots\times\P^1}_{n}\,$, and fix the fan $\fan$ to be
the one associated to $(\P^1)^n$, as described in Example~\ref{ex:P1n}.
We set the coordinate as
\[
(\P^1)^n = \Bigl\{ ([x_1;y_1],\cdots,[x_n;y_n])\ \Bigl\vert\ [x_i;y_i]\in\P^1 \text{ for $i=1,\cdots,n$}\Bigr\}.
\]

\subsection{Pulling back divisors by monomial maps on $(\P^1)^n$}
Let $D_i$, $E_i$ be the divisors defined by the equations $x_i=0$, $y_i=0$, respectively,
for $i=1,\cdots,n$.
Notice that if we set $\tau_i=\R_{\ge 0}\cdot e_i$, then $D_i=V(\tau_i)$ and $E_i=V(-\tau_i)$.
The group of torus invariant divisors are then generated by the $D_i$'s and $E_i$'s,
that is,
\[
\CDivT((\P^1)^n)\cong\Z^{2n}=\Z D_1\oplus\Z E_1\oplus\cdots\oplus\Z D_n\oplus \Z E_n.
\]

Notice that the support function for $D_i$ and $E_i$ are
\begin{align*}
\psi_{D_i}(a_1,\cdots,a_n) &= \min\{-a_i,0\},\\
\psi_{E_i}(a_1,\cdots,a_n) &= \min\{a_i,0\}.
\end{align*}

Consider the monomial map $f_A$ on $(\P^1)^n$ associated to a matrix $A=(a_{ij})\in \M_n(\Z)$.
Using Theorem~\ref{thm:pullback} and the formula for $\psi_{D_i}$ and $\psi_{E_i}$,
we obtain the following explicit description of the pull back $f_A^*$ on the group $\CDivT((\P^1)^n)$.
First, fix the ordered basis $D_1,E_1,\cdots,D_n,E_n$. With respect to this basis, $f_A^*$ is
represented by a $2n\times 2n$ integer matrix $(\alpha_{ij})_{i,j=1\cdots,n}$.
Here each $\alpha_{ij}$ is a $2\times 2$ block:
\begin{equation}
\label{eq:P1n}
\alpha_{ij}= \begin{cases} \begin{pmatrix}|a_{ji}| & 0 \\ 0 & |a_{ji}|\end{pmatrix}
                           &   \text{if  $a_{ji}\ge 0$,}  \\       \mbox{}\\
                       \begin{pmatrix}0 & |a_{ji}| \\  |a_{ji}|& 0 \end{pmatrix}
                           &   \text{if  $a_{ji}\le 0$.}\end{cases}
\end{equation}
Notice that the block $\alpha_{ij}$ corresponds to the entry $a_{ji}$, not $a_{ij}$. This is because
when we induce the map $f_A$ from $A$, the matrix $A$ acts on the lattice $N$, and we have to
take its transposition $^tA$
to act on $M=N^\vee$, see Examples 2.1 and 2.2.

Passing to the Picard groups, $D_i$ is linearly equivalent to $E_i$, so
\[
\Pic((\P^1)^n)\cong\Z^n=\Z\cdot [D_1]\oplus\cdots\oplus\Z\cdot [D_n].
\]
The pull-back $f_A^*$ on Picard group, with respect to the basis $[D_1],\cdots,[D_n]$,
is given by the matrix $\bigl(|a_{ji}|\bigr)$, i.e., each entry of $f_A^*$ is the absolute value
of the corresponding entry of the transpose matrix $^tA$. This observation about $f^*$ has an
immediate application, which is shown in the next subsection.

\subsection{The first dynamical degree of a monomial map}
\label{sec:d1_mono}
It is known that the first dynamical degree of a monomial map $f_A$ is the spectral radius of the matrix $A$
(see ~\cite[Theorem 6.2]{HP}). In this section we will provide an alternative proof of this fact.

For a compact K\"{a}hler manifold $X$, and a rational self map $f:X\to X$, let
$f^*: H^{1,1}(X;\R)\to H^{1,1}(X;\R)$ be the pull back map on the $(1,1)$ cohomology group.
Put any norm $\norm{\cdot}$ on the space $\End_\R(H^{1,1}(X;\R))$ of
$\R$-linear endomorphism on $H^{1,1}(X;\R)$.
Then the first dynamical degree of $f$, denoted by $\delta_1(f)$, is defined by
\[
\delta_1(f)=\liminf_{k\to\infty}\, \norm{(f^k)^*}^{1/k}.
\]
A property of the first dynamical degree is that it is invariant under birational conjugate
(see ~\cite[Proposition 2.6 and Corollaire 2.7]{Guedj}).
Therefore, it is a common method to find a good birational model $\tilde{X}$ of $X$ so that it is easier
to compute the dynamical degree for the conjugate $\tilde{f}:\tilde{X}\to\tilde{X}$.
In this section, we would like to use the model $(\P^1)^n$. With this model, we can obtain another
proof of the following:

\begin{thm}
\label{thm:d1_mono}
The first dynamical degree of a monomial map $f_A$ is the spectral radius of the matrix $A$.
\end{thm}

\begin{proof}
For $X=(\P^1)^n$,
we have $H^{1,1}(X;\R)\cong\Pic(X)\otimes_\Z \R$. Also, remember that $f_A^*$ can be represented by the
matrix $\bigl(|a_{ji}|\bigr)$.
For a linear map in $\End_\R(H^{1,1}(X;\R))$, we first represent it by a matrix with
respect to the ordered basis $[D_1],\cdots,[D_n]$. Then we take
the $L^1$ norm of the matrix representation. This gives a norm on $\End_\R(H^{1,1}(X;\R))$, i.e., we set
\[
\norm{f_A^*}=\norm{\bigl(|a_{ji}|\bigr)}_1=\sum_{i,j=1}^n |a_{ji}|.
\]
It is easy to see that
\begin{equation}
\label{eq:fA*}
\norm{f_A^*}=\norm{A}_1.
\end{equation}
Furthermore, for the
$k$-th iterate of $f_A$, we know $f_A^k = f_{A^k}$.
By substituting $A$ with $A^k$ in~(\ref{eq:fA*}),
we also know $\norm{(f_A^k)^*}=\norm{(f_{A^k})^*}=\norm{A^k}_1$. Therefore,
\begin{align*}
\delta_1(f)&=\liminf_{k\to\infty}\, \norm{(f_A^k)^*}^{1/k}\\
           &=\liminf_{k\to\infty}\, \norm{A^k}_1^{1/k}\\
           &=\rho(A).
\end{align*}
\end{proof}

\subsection{Algebraic stability of monomial maps on $(\P^1)^n$}
We now turn to the discussion about algebraic stability.
The goal of this section is to show that for monomial maps on $(\P^1)^n$,
being algebraic stable is equivalent to
being strongly algebraic stable.

Let $A=(a_{ij}),B=(b_{ij})$ be two $n\times n$ integer matrices, and $f_A,f_B$ be the monomial maps on $(\P^1)^n$
induced by $A,B$, respectively. Also, let $C=AB$ be their product, and assume that $C=(c_{ij})$.
Then $f_A\circ f_B=f_{AB}=f_C$, thus we know on the Picard group, the pull back $(f_A\circ f_B)^*=f_C^*$ is represented
by the matrix whose $(i,j)$-th entry is given by $|c_{ji}|=|\sum_{k=1}^n a_{jk}b_{ki}|$. On the other hand, the
$(i,j)$-th entry of the matrix representing $f_B^*\circ f_A^*$ is
\[
\sum_{k=1}^n |b_{ki}|\cdot|a_{jk}|=\sum_{k=1}^n |a_{jk}b_{ki}|.
\]
The two numbers are equal if and only if all the non-zero summands $~a_{jk}b_{ki}~$
have the same sign, i.e., they are either all positive or all negative. Thus we know that
$(f_A\circ f_B)^*=f_B^*\circ f_A^*$ for divisor classes in the Picard group if and only if the following
condition holds:

\begin{itemize}
\item[$(\star)$] For each entry $\sum_{k=1}^n a_{jk}b_{ki}$ of the matrix $AB$, all the non-zero summands
$a_{jk}b_{ki}$ have the same sign.
\end{itemize}

The situation in the group $\CDivT((\P^1)^n)$ is similar, but more complicated. We know that for the fixed ordered basis
$D_1,E_1,\cdots,D_n,E_n$, the pull back maps $f_A^*$, $f_B^*$, and $(f_A\circ f_B)^*=f_C^*$ on $\CDivT(X)$ are
represented by the $2n\times 2n$ integer matrices $(\alpha_{ij})$, $(\beta_{ij})$, and $(\gamma_{ij})$, respectively.
Here each of the $\alpha_{ij}$, $\beta_{ij}$, and $\gamma_{ij}$ is a $2\times 2$ block as described in (\ref{eq:P1n}).
The composition $f_B^*\circ f_A^*$ is then represented by the matrix product $(\beta_{ij})\cdot(\alpha_{ij})$. We can
do the matrix multiplication by blocks, and assume $(\gamma'_{ij})=(\beta_{ij})\cdot(\alpha_{ij})$, where
\[
\gamma'_{ij}=\sum_{k=1}^n \beta_{ik}\cdot\alpha_{kj}.
\]
A simple calculation shows that
\[
\beta_{ik}\cdot\alpha_{kj}= \begin{cases} \begin{pmatrix}|a_{jk}b_{ki}| & 0 \\ 0 & |a_{jk}b_{ki}|\end{pmatrix}
                             &   \text{if  $a_{jk}b_{ki}\ge 0$,}  \\       \mbox{}\\
                           \begin{pmatrix}0 & |a_{jk}b_{ki}| \\  |a_{jk}b_{ki}|& 0 \end{pmatrix}
                             &   \text{if  $a_{jk}b_{ki}\le 0$.}\end{cases}
\]
Since the blocks $\gamma_{ij}$ are either of the form $(\begin{smallmatrix}\ell & 0 \\ 0 & \ell\end{smallmatrix})$
or of the form $(\begin{smallmatrix}0 & \ell \\ \ell & 0\end{smallmatrix})$ for some non-negative integer $\ell$,
a necessary condition for
$\gamma'_{ij}=\gamma_{ij}$ is that all block summands $\beta_{ik}\cdot\alpha_{kj}$ are of the same form, i.e.,
all of the form  $(\begin{smallmatrix}\ell & 0 \\ 0 & \ell\end{smallmatrix})$ or all of the form
$(\begin{smallmatrix}0 & \ell \\ \ell & 0\end{smallmatrix})$.
This is equivalent to the condition that all the nonzero terms $a_{jk}b_{ki}$ are of the same sign for $k=1,\cdots,n$.
On the other hand, once we have that all the nonzero terms $a_{jk}b_{ki}$ are of the same sign, it is easy to see that
we will automatically have $\gamma'_{ij}=\gamma_{ij}$. Therefore, we conclude that $(f_A\circ f_B)^*=f_B^*\circ f_A^*$
as maps on $\CDivT(X)$ if and only if the condition $(\star)$ holds again. We can summarize the above discussion as
the following proposition.

\begin{prop}
$(f_A\circ f_B)^*=f_B^*\circ f_A^*$ for $T$-invariant divisors if and only if $(f_A\circ f_B)^*=f_B^*\circ f_A^*$
for divisor classes in the Picard group, if and only if the condition $(\star)$ is satisfied.\hfill\qed
\end{prop}

%%%%%%%%%%%%%%%%%%%%%%%%%%%%%%%%%%
\subsection{Stability of rational maps on $(\P^1)^n$}

In the last two subsections of this paper, we will leave the realm of monomial maps, and look at some general
phenomenon about algebraic stability of rational maps on $(\P^1)^n$.
Assume that we have a rational map $f:\C^n\dashrightarrow\C^n$ given by
\[
f(\z)=f(z_1,\cdots,z_n) = \Bigl(\frac{p_1(\z)}{q_1(\z)},\cdots,\frac{p_n(\z)}{q_n(\z)}\Bigr),
\]
where the $p_j, q_j$ are polynomials in $\z=(z_1,\cdots,z_n)$ for $i=1,\cdots,n$.
We can assume that $p_j$ and $q_j$ are pairwise relatively prime, otherwise we can divide them
by their greatest common divisor.
The map $f$ induces a rational map, also denoted by $f$, on $(\P^1)^n$, in the following way.
\begin{align*}
   & f([x_1;y_1],\cdots,[x_n;y_n]) \\
=\ & \Bigl( [P_1(x_i,y_i);Q_1(x_i,y_i)],\cdots,[P_n(x_i,y_i);Q_n(x_i,y_i)]\Bigr).
\end{align*}
Here the $P_j$ and $Q_j$ are obtained by homogenizing the polynomials $p_j, q_j$ in the $j$-th component of $f$,
with respect to every pair of variables $(x_i,y_i)$, by setting $z_i=x_i/y_i$. We will use $P_1(x_i,y_i)$
as a shorthand for $P_1(x_1,y_1,\cdots,x_n,y_n)$. The concrete formulae for $P_j$ and $Q_j$ are
\begin{align*}
P_j(x_i,y_i) &= P_j(x_1,y_1,\cdots,x_n,y_n) \\
             &= \Bigl(\prod_{i=1}^n y_i^{\max \{ \deg_{z_i}(p_j),\deg_{z_i}(q_j) \}}\Bigr)\cdot p_j(\frac{x_1}{y_1},\cdots,\frac{x_n}{y_n}),\\
Q_j(x_i,y_i) &= Q_j(x_1,y_1,\cdots,x_n,y_n) \\
             &= \Bigl(\prod_{i=1}^n y_i^{\max \{ \deg_{z_i}(p_j),\deg_{z_i}(q_j) \}}\Bigr)\cdot q_j(\frac{x_1}{y_1},\cdots,\frac{x_n}{y_n}).
\end{align*}
Thus the polynomials $P_j$ and $Q_j$ are homogeneous in each pair of variables $(x_i,y_i)$, of the same degree
$=\max \{ \deg_{z_i}(p_j),\deg_{z_i}(q_j) \}$. We denote this degree by
$\deg_{(x_i,y_i)} P_j =  \deg_{(x_i,y_i)} Q_j$,
and call such polynomials $P_j$ and $Q_j$ {\em multi-homogeneous}.

Conversely, given $~2n~$ multi-homogeneous polynomials
 $P_j(x_i,y_i)$, $Q_j(x_i,y_i)$, $j=1,\cdots,n$, which are pairwise relatively prime, and satisfy
$\deg_{(x_i,y_i)} P_j =  \deg_{(x_i,y_i)} Q_j$ for all $i,j=1,\cdots,n$. Then they induce a rational map
$f: (\P^1)^n \dashrightarrow (\P^1)^n$ by sending $([x_1;y_1],\cdots,[x_n;y_n])$ to
\[
([P_1(x_i,y_i);Q_1(x_i;y_i)],\cdots,[P_n(x_i,y_i);Q_n(x_i;y_i)]).
\]
The indeterminacy set of the rational map $f$ is given by $I_f=\cup_{j=1}^n I_{f,j}$, where $I_{f,j}$ is
the set defined by the equations $P_j=Q_j=0$.

Recall that the Picard group of $(\P^1)^n$ is
\[
\Pic((\P^1)^n)= \Z\cdot [D_1]\oplus\cdots\oplus\Z\cdot [D_n],
\]
where $D_i$ is the divisor defined by $x_i=0$, and $[D_i]$ is the linear equivalence class
of $D_i$ in $\Pic((P^1)^n)$.
Therefore, the pull back map $f^*:\Pic((\P^1)^n)\to \Pic((\P^1)^n)$,
with respect to the ordered basis
$[D_1],\cdots,[D_n]$, is represented by the matrix
\begin{align*}
\Deg(f) &= \Bigl(\max \{ \deg_{z_i}(p_j),\deg_{z_i}(q_j) \}\Bigr)_{1\le i\le n; 1\le j\le n.} \\
        &= \Bigl(\deg_{(x_i,y_i)}(P_j)\Bigr)_{1\le i\le n; 1\le j\le n.}\\
        &= \Bigl(\deg_{(x_i,y_i)}(Q_j)\Bigr)_{1\le i\le n; 1\le j\le n.}\\
\end{align*}
Therefore, the condition $(f^*)^n=(f^n)^*$ for algebraic stability can be translated into the condition $\Deg(f)^n=\Deg(f^n)$.
We will give a geometric characterization
of algebraic stable maps on $(\P^1)^n$.
Before doing that, we need to introduce some notations and facts about $(\P^1)^n$.

Suppose we equip $(\C^2)^n$ with the coordinate $(x_1,y_1,\cdots,x_n,y_n)$,
and let $E_j = \{x_j=y_j=0\}\subset (\C^2)^n$, then there is a quotient map
\begin{align*}
\pi : (\C^2)^n\backslash (\cup_{j=1}^n E_j) & \longrightarrow  (\P^1)^n \\
      (x_1,y_1,\cdots,x_n,y_n) & \longmapsto  ( [x_1;y_1],\cdots,[x_n;y_n])
\end{align*}
For each point ${\bf x}\in(\P^1)^n$, the fibre $\pi^{-1}({\bf x})$ is an algebraic torus $(\C^*)^n$.

Suppose that a rational map $f:(\P^1)^n \dashrightarrow(\P^1)^n$ is given by $P_j$ and $Q_j$, as described above.
We can {\em lift} the rational map $f$ to obtain a polynomial map $F:(\C^2)^n\to(\C^2)^n$.
$F$ is defined by the same polynomials $P_j$ and $Q_j$ as $f$, i.e.,
\[
F(x_i,y_i)=\bigl(P_1(x_i,y_i), Q_1(x_i,y_i),\cdots, P_n(x_i,y_i), Q_n(x_i,y_i) \bigr).
\]
Notice that, a point ${\bf x}\in(\P^1)^n$ is in the indeterminacy set $I_f$ if and only if
$F(\pi^{-1}({\bf x}))\subset (\cup_{j=1}^n E_j)$. When this happens, since $\pi^{-1}({\bf x})$ is
irreducible (in the Zariski topology), we must have $F(\pi^{-1}({\bf x}))\subset E_j$ for some $j$.
To conclude, we have
\[
{\bf x}\in I_f ~\Longleftrightarrow~ F(\pi^{-1}({\bf x}))\subset E_j \text{ for some $j$.}
\]

A hypersurface $V\subset (\P^1)^n$ is defined by a multi-homogeneous polynomial $\varphi=\varphi(x_1,y_1,\cdots,x_n,y_n)=0$.
We can consider the
{\em lifting of} $V$ in $(\C^2)^n$, defined by $\widetilde{V} = \overline{\pi^{-1}(V)}$.
$\widetilde{V}$ is a hypersurface in $(\C^2)^n$, and the defining equation for $\widetilde{V}$ is also $\varphi=0$.
Notice that $V$ is irreducible in the Zariski topology on $(\P^1)^n$ if and only if
$\widetilde{V}$ is irreducible in the Zariski topology on $(\C^2)^n$. This is because if we can factor
$\varphi = \varphi_1\cdot\varphi_2$, then both $\varphi_1$ and $\varphi_2$ have to be multi-homogeneous.

The following proposition and theorem characterize, geometrically, the algebraic stable maps on $(\P^1)^n$.
The proof is a modification of the method used to prove a similar proposition
on $\P^n$ by Fornaess and Sibony (\cite{FS}, see also~\cite[Proposition 1.4.3]{Sibony}).
Also, the results were already given, in the more general context of multiprojective spaces,
by Favre and Guedj~\cite[Proposition 1.7]{FG}. We include them here for completeness.

\begin{prop}
\label{prop:rat_stable}
For two rational maps $f,g: (\P^1)^n\dashrightarrow (\P^1)^n$, the relation
$\Deg(f\circ g)=\Deg(g)\cdot\Deg(f)$ holds if and only if there is no hypersurface $V\subset (\P^1)^n$ such that
$g(V\backslash I_g)\subset I_f$.
\end{prop}

\begin{proof}

First, if there is such a $V$, we can assume that $V$ is irreducible. Then $U=\pi^{-1}(V\backslash I_g)$ is a nonempty
open subset of $\widetilde{V}$, hence is dense in $\widetilde{V}$ and is irreducible. The condition
$g(V\backslash I_g)\subset I_f$ means that for all $y\in U$, we have $F(G(y))\in E_j$ for
some $j$. A priori the $E_j$ may depend on $y$, but since $U$ is irreducible, this implies that $F(G(U))\subset E_j$
for some $j$. Without loss of generality, assume $j=1$.
Furthermore, since $U$ is open and dense in $\widetilde{V}$, and $E_1$ is a closed subset of $(\P^1)^n$,
we conclude that $F(G(\widetilde{V}))\subset E_1$ as well.

Suppose $V$ is defined by the multi-homogeneous polynomial $\varphi$, and for ${\bf x}\in(\P^1)^n$,
the maps $f,g$ are given by
\begin{align*}
f({\bf x}) &= \Bigl([P_1({\bf x});Q_1({\bf x})],\cdots,[P_n({\bf x});Q_n({\bf x})]\Bigr),\\
g({\bf x}) &= \Bigl([P'_1({\bf x});Q'_1({\bf x})],\cdots,[P'_n({\bf x});Q'_n({\bf x})]\Bigr).
\end{align*}
The $j$-th component of the composition map $f\circ g$ is given by the polynomials
$P''_j=P_j(P'_1,Q'_1,\cdots,P'_n,Q'_n)$ and $Q''_j=Q_j(P'_1,Q'_1,\cdots,P'_n,Q'_n)$.
A computation on degree shows that
\[
\begin{split}
\deg_{(x_i,y_i)}(P''_j) &= \deg_{(x_i,y_i)} (Q''_j)\\
                      &= \sum_{k=1}^n \deg_{(x_i,y_i)}(P'_k) \cdot \deg_{(x_k,y_k)}(P_j).
\end{split}
\]
This is the $(i,j)$-th component of the product of matrices $\Deg(g)\cdot\Deg(f)$.
On the other hand, $F(G(\widetilde{V}))\subset E_1$ implies that $\varphi$ divides both
polynomials $P''_1$ and $Q''_1$. Thus, for some $i$ such that $\deg_{(x_i,y_i)}(\varphi)>0$,
the $(i,1)$-th component of the matrix $\Deg(f\circ g)$ will be strictly less than
the $(i,1)$-th component of the product of matrices $\Deg(g)\cdot\Deg(f)$. The two matrices cannot be equal.

Conversely, it is easy to see that if there is no such hypersurface, then the polynomials
$P''_j$ and $Q''_j$ will be pairwise relatively prime, with the desired degrees. Hence we will have
$\Deg(f\circ g)=\Deg(g)\cdot\Deg(f)$.
\end{proof}

\begin{thm}
A rational map $f : (\P^1)^n\dashrightarrow (\P^1)^n$ is algebraically stable if and only if
there does not exist an integer $k$ and a hypersurface $V\subset (\P^1)^n$ such that
$f^k(V\backslash I_{f^k})\subset I_f$.
\end{thm}

\begin{proof}
This is a direct consequence of the Proposition~\ref{prop:rat_stable} by using an induction
argument on $k$ and setting $g=f^k$ in the proposition.
\end{proof}

\subsection{Stability of polynomial maps on $(\P^1)^n$}
\label{subsec:stab_poly}

Recall that for $f(\z)=({p_1(\z)}/{q_1(\z)},\cdots,{p_n(\z)}/{q_n(\z)})$,
it induces a rational map on $(\P^1)^n$, and the pull back $f^*$ is represented by the matrix
\[
\Deg(f) = \Bigl(\max \{ \deg_{z_i}(p_j),\deg_{z_i}(q_j) \}\Bigr)_{1\le i\le n; 1\le j\le n.}
\]
In particular, if $f=(f_1,\cdots,f_n)$ is a polynomial map, then $p_j=f_j$ and
$\deg_{z_i}(q_j)=0$ for all $i,j$, hence
\[
\Deg(f)=\Bigl(\deg^+_{z_i}(f_j)\Bigr).
\]
Here $\deg^+_{z_i}(f_j)=\max \{ \deg_{z_i}(f_j),0 \}$ is almost the degree of $f$; the only difference is that
for the zero polynomial, we have $\deg^+_{z_i}(0)=0$ instead of the usual convention $\deg_{z_i}(0)=-\infty$.

Given a polynomial map $f$, for notational simplicity, we will denote $\deg^+_{z_i}(f_j)$ by $\deg_i(f_j)$
and just call it {\em the degree of} $f_j$ with respect to $z_i$, and
we will call $\Deg(f)$ the {\em degree matrix} of $f$.

Our next goal is to proof the following Theorem~\ref{thm:poly_stable}, which gives a family of algebraically stable
polynomial maps on $(\P^1)^n$, and a partial converse which gives a characterization for
algebraically stable polynomial maps on $(\P^1)^n$ under certain condition. First, we need to define some
terminologies.

\begin{defn}
For a polynomial $h\in\C[z_1,\cdots,z_n]$ and a monomial $\mu=z_1^{a_1}\cdots z_n^{a_n}$, we said that
$\mu$ is a {\em monomial term} of $h$ if the coefficient of $\mu$ in $h$ is not zero.
A monomial term $\mu$ of $h$ is said to be {\em the dominant term} of  $h$ if for all $i=1,\cdots,n$,
we have $\deg_i(h)=\deg_i(\mu)$.
\end{defn}

Equivalently, a monomial term $\mu=z_1^{a_1}\cdots z_n^{a_n}$ of $h$ is the dominant term of  $h$
if and only if,
for all monomial term $z_a^{b_1}\cdots z_n^{b_n}$ of $h$ , we have $a_i\ge b_i$ for $i=1,\cdots,n$.
For example, the polynomial $h=2z_1^2z_2+3z_1^2+z_1z_2-5z_2-1$ has a dominant term $z_1^2z_2$. Notice that
not all polynomials have a dominant term. For example, $h=z_1+z_2$ does not have a dominant term.

\begin{thm}
\label{thm:poly_stable}
Let $f=(f_1,\cdots,f_n)$ be a polynomial map.
\begin{enumerate}
\item
If each $f_j$ is dominated by a monomial term, then $f$ is
algebraically stable on $(\P^1)^n$.
\item
Assume that, for some iterate $f^N=(f_1^{(N)},\cdots,f_n^{(N)})$ of $f$, we have $\deg_{z_i}(f^{(N)}_j)>0$  for all $i,j=1,\cdots,n$.
Then $f$ being algebraically stable on $(\P^1)^n$ implies that each $f_j$ must have a dominant term.
\end{enumerate}
\end{thm}

We will prove the theorem in steps. First, observe that, if each $f_j$ has a dominant term
$\mu_j=z_1^{a_{1j}}\cdots z_n^{a_{nj}}$, then we know $\deg_i(f_j)=a_{ij}$, and therefore
$\Deg(f)=(a_{ij})$.

Next, if $f=(f_1,\cdots,f_n)$ and $g=(g_1,\cdots,g_n)$ are two polynomial maps, such that each of the
$f_j$ and $g_k$ has a dominant term, say $\mu_j=z_1^{a_{1j}}\cdots z_n^{a_{nj}}$
and $\nu_k= z_1^{b_{1k}}\cdots z_n^{b_{nk}}$, respectively.
Consider the $j$-th component of the composition $fg=f\circ g$. It will be of the following form:
\begin{align*}
(fg)_j &= f_j(g_1,\cdots,g_n) \\
       &= c_j\cdot\mu_j(\nu_1,\cdots,\nu_n)+\{\text{lower order terms}\}\\
       &= c_j\cdot\nu_1^{a_{1j}}\cdots \nu_n^{a_{nj}}+\{\text{lower order terms}\}\\
       &= c_j\cdot\prod_{k=1}^n\Bigl(z_1^{b_{1k}}\cdots z_n^{b_{nk}}\Bigr)^{a_{kj}}+\{\text{lower order terms}\}
\end{align*}
where $c_j$ is some constant. That is, $\mu_j(\nu_1,\cdots,\nu_j)$ is the dominant term
of $(fg)_j$, and the degree
\[
\deg_i((fg)_j)=\sum_{k=1}^n b_{ik}a_{kj},
\]
which is the $(i,j)$-th component of the product of matrices $\Deg(g)\cdot\Deg(f)$.
We summarize the above discussion as the following proposition.

\begin{prop}
If $f=(f_1,\cdots,f_n)$ and $g=(g_1,\cdots,g_n)$ are two polynomial maps, such that each of the
$f_j$ and $g_j$ has a dominant term, then the composition $(fg)$ is also a polynomial map such that each
component has a dominant term. Furthermore, for the degree matrix of $(fg)$, we have
\[
\Deg(fg)=\Deg(g)\cdot\Deg(f).
\]
\hfill\qed
\end{prop}

As a corollary of the proposition, we can now prove the first part of Theorem~\ref{thm:poly_stable}.

\begin{cor*}
If $f=(f_1,\cdots,f_n):\C^n\to\C^n$ is a polynomial map, and if each $f_j$ has a dominant term, then
$f$ is algebraically stable on $(\P^1)^n$.
\end{cor*}
\begin{proof}
We know that for all $k$, the iterate $f^k$ is also a polynomial map such that every component
has a dominant term. Hence an induction argument shows that
\[
\Deg(f^k)=(\Deg(f))^k.
\]
Since the degree matrix represents $f^*$ with respect to the ordered basis $[D_1],\cdots,[D_n]$,
we also have
\[
(f^k)^*=(f^*)^k.
\]
\end{proof}

\begin{cor*}
If $f=(f_1,\cdots,f_n):\C^n\to\C^n$ is a polynomial map, and if each $f_j$ has a dominant term, then
the first dynamical degree of $f$ is an algebraic integer.
\end{cor*}
\begin{proof}
We know that $f$ is algebraically stable, and $f^*$ is represented by the degree matrix $\Deg(f)$.
The first dynamical degree of $f$ is the spectral radius of the degree matrix, an integer matrix.
Hence  the first dynamical degree is an algebraic integer.
\end{proof}
There is a conjecture proposed by Bellon and Viallet (see ~\cite[Conjecture 1.1]{HP}),
namely, the first dynamical degree of every rational map is an algebraic integer.
The corollary shows that the conjecture holds for the case of polynomial maps with
a dominant term for each component.

Notice that every monomial polynomial has a dominant term, namely the monomial itself.
Hence we obtain another proof that every monomial polynomial map is algebraically stable on $(\P^1)^n$.

Now we turn to the second part of Theorem~\ref{thm:poly_stable}.

\begin{prop}
Let $f_1,\cdots,f_n,g_1,\cdots g_n\in \C[z_1,\cdots,z_n]$ be polynomials
such that $\deg_i(g_j)> 0$ for all $i,j=1,\cdots,n$. They induce polynomial maps
$f=(f_1,\cdots,f_n)$ and $g=(g_1,\cdots,g_n):\C^n\to\C^n$. If we have
\[
\Deg(f g)=\Deg(g)\cdot\Deg(f),
\]
then every component $f_j$ must have a dominant monomial term.
\end{prop}

\begin{proof}
Assume otherwise, i.e., some $f_j$ does not have a dominant term, we want to show the two
matrices are different. Under the assumption, for every monomial term
$\mu=z_1^{a_1}\cdots x_n^{a_n}$, there is some $a_k<\deg_k(f_j)$; without loss of generality
we can assume $a_1<\deg_1(f_j)$.
Consider the composition
$\mu\circ g=\mu(g_1,\cdots,g_n)$, it is easy to see that, for all $i=1,\cdots,n$,
\begin{align*}
\deg_i(\mu\circ g) &= a_1\cdot \deg_i(g_1) + \cdots + a_n\cdot\deg_i(g_n)\\
                   &< \deg_1(f_j)\cdot \deg_i(g_1) + \cdots + \deg_n(f_j)\cdot\deg_i(g_n)\\
\end{align*}
The inequality is strict because $a_1<\deg_1(f_j)$ and $\deg_i(g_1)>0$. Therefore, we know
\[
\deg_i(f_j\circ g)< \deg_1(f_j)\cdot \deg_i(g_1) + \cdots + \deg_n(f_j)\cdot\deg_i(g_n),
\]
Notice that the term on the left is the $(i,j)$-th entry of the matrix $\Deg(f g)$, whereas the term
on the right is the $(i,j)$-th entry of the matrix $\Deg(g)\cdot\Deg(f)$.
Hence, the two matrices are different.
\end{proof}

The condition $\deg_i(g_j)> 0$ in the proposition is essential. For example, let
$f(z_1,z_2,z_3)=g(z_1,z_2,z_3)=(z_2,z_3,z_1+z_2)$, then we still have
$\Deg(f g)=\Deg(g)\cdot\Deg(f)$, but $f_3=z_1+z_2$ does not have a dominant monomial term.

\begin{cor*}
Suppose that $f=(f_1,\cdots,f_n):\C^n\to\C^n$ is a polynomial map, and for some iterate
$f^N=(f_1^{(N)},\cdots,f_n^{(N)})$ of $f$,
we have $\deg_i(f_j^{(N)})>0$ for all $i,j=1,\cdots,n$.
If $\Deg(f^k)=\Deg(f)^k$ for all $k$, then every component $f_j$ of $f$ must have a dominant monomial term.
\end{cor*}

\begin{proof}
Apply $f$ and $g=f^N$ to the proposition.
\end{proof}

This also concludes the proof of Theorem~\ref{thm:poly_stable}.\hfill\qed

%%%%%%%%%%%%%%%%%%%%%%%%%%%%%%%%%%%%%%%%%%%%%%%%%%%%%%%%%%%%%%%%%%%%%%
%%%%
%%%%  Bibliography
%%%%
%%%%%%%%%%%%%%%%%%%%%%%%%%%%%%%%%%%%%%%%%%%%%%%%%%%%%%%%%%%%%%%%%%%%%%

\end{document}